%% file: online_ramsey.tex
\documentclass[10pt,final,a4paper]{article}

\usepackage[T1]{fontenc}

\usepackage[english]{babel}

\usepackage[utf8]{inputenc}

\usepackage[sc]{mathpazo}

\usepackage{amsmath,amssymb,amsfonts,mathrsfs,amsthm}

\usepackage{xparse}

\usepackage{aliascnt}

\usepackage[affil-it]{authblk}

\usepackage{datetime}

\usepackage{mathtools}

\usepackage[activate]{pdfcprot}

\usepackage{thmtools,thm-restate}

\declaretheoremstyle[notefont=\normalfont\bfseries,bodyfont=\normalfont\itshape]{mystyle}
\declaretheorem[style=mystyle,numberwithin=section]{theorem}
\declaretheorem[style=mystyle,sibling=theorem]{lemma}
\declaretheorem[style=mystyle,sibling=theorem]{definition}

\declaretheorem[style=mystyle,sibling=theorem]{claim}

\usepackage[linkcolor=black,colorlinks=true,citecolor=black,filecolor=black]{hyperref}

\newcommand{\N}{\mathbb{N}}

\DeclarePairedDelimiter\abs{\lvert}{\rvert}
\DeclarePairedDelimiter\norm{\lVert}{\rVert}
\DeclarePairedDelimiter\floor{\lfloor}{\rfloor}
\DeclarePairedDelimiter\ceil{\lceil}{\rceil}

\makeatletter
\let\oldabs\abs
\def\abs{\@ifstar{\oldabs}{\oldabs*}}
\makeatother

\makeatletter
\let\oldnorm\norm
\def\norm{\@ifstar{\oldnorm}{\oldnorm*}}
\makeatother

\makeatletter
\let\oldfloor\floor
\def\floor{\@ifstar{\oldfloor}{\oldfloor*}}
\makeatother

\makeatletter
\let\oldceil\ceil
\def\ceil{\@ifstar{\oldceil}{\oldceil*}}
\makeatother

\renewcommand{\epsilon}{\ensuremath\varepsilon}

\renewcommand{\phi}{\ensuremath{\varphi}}

\let\originalleft\left
\let\originalright\right
\renewcommand{\left}{\mathopen{}\mathclose\bgroup\originalleft}
\renewcommand{\right}{\aftergroup\egroup\originalright}

\DeclareMathOperator*{\argmax}{arg\,max}

\newcommand{\ccint}[2] { \left[ {#1},{#2} \right] }

\newcommand{\set}[1] { \left[ {#1} \right] }

\renewcommand*{\Pr}[1]{\mathrm{Pr}\left[#1\right]}
\newcommand*{\E}[1]{\mathrm{E}\left[#1\right]}
\newcommand{\Var}[1]{ \mbox{Var} \left[ {#1} \right] }
\NewDocumentCommand\G{gggg}{%
	\IfNoValueTF{#1}{\mathcal{G}}{%
	\IfNoValueTF{#2}{\mathcal{G}\left({#1}\right)}{%
    \IfNoValueTF{#3}{\mathcal{G}\left({#1}, {#2}\right)}{%
    \IfNoValueTF{#4}{\mathcal{G}\left({#1}, {#2}, {#3}\right)}{%
				\mathcal{G}\left({#1}, {#2}, {#3}, {#4}\right)}}}}%
}

\NewDocumentCommand\reg{gg}{%
	\IfNoValueTF{#1}{\reg{\epsilon}}{%
	\IfNoValueTF{#2}{\(\left({#1}\right)\)-regular}{\(\left({#1}, {#2}\right)\)-regular}}%
}
\NewDocumentCommand\lreg{gggg}{%
	\IfNoValueTF{#1}{\lreg{\epsilon}}{%
	\IfNoValueTF{#2}{\(\left({#1}\right)\)-lower-regular}{%
    \IfNoValueTF{#3}{\(\left({#1}, {#2}\right)\)-lower-regular}{%
    \IfNoValueTF{#4}{\(\left({#1}, {#2}, {#3}\right)\)-lower-regular}{\(\left({#1}, {#2}, {#3}, {#4}\right)\)-lower-regular}}}}%
}

\newcommand*{\uf}[2]{\(\left({#1},{#2}\right)\)-upper-uniform}
\newcommand*{\ue}[2]{\(\left({#1},{#2}\right)\)-upper-extensible}
\renewcommand*{\O}[1]{O\left({#1}\right)}
\renewcommand*{\o}[1]{o\left({#1}\right)}
\newcommand*{\Om}[1]{\Omega\left({#1}\right)}
\newcommand*{\om}[1]{\omega\left({#1}\right)}
\newcommand*{\The}[1]{\Theta\left({#1}\right)}
\newcommand*{\eps}{\epsilon}

\title{Online Ramsey Games for more than two colors}
\author{Andreas Noever\footnote{author was supported by grant no.\ 200021 143338  of the Swiss National Science Foundation.}}
\affil{Institute of Theoretical Computer Science \\ ETH Zürich, 8092 Zürich, Switzerland \\ anoever@inf.ethz.ch}
\date{}

\begin{document}

%\frontmatter

%% Title page is autogenerated from document information above.  DO
%% NOT CHANGE.
%\begin{titlingpage}
%  \calccentering{\unitlength}
%  \begin{adjustwidth*}{\unitlength-24pt}{-\unitlength-24pt}
    \maketitle
%  \end{adjustwidth*}
%\end{titlingpage}

%% The abstract of your thesis.  Edit the file as needed.
\input{abstract}

%% TOC with the proper setup, do not change.
%\cleartorecto
%\tableofcontents
%\mainmatter
\renewcommand{\sectionautorefname}{Section}
\renewcommand{\subsectionautorefname}{Section}
\renewcommand{\subsubsectionautorefname}{Section}
%% Your real content!
%\input{introduction}

\section{Introduction}
	\input{introduction}
	\subsection{Preliminaries and Notation}
		\input{preliminaries}

	\subsubsection{Szemer\'edi's regularity lemma for sparse graphs}
		\input{regularitylemma}
	\subsubsection{A K\L R type statement for rooted graphs}\label{section:klr statement}
		\input{intro_klr}
\section{Proof of Main Theorem}\label{section:main proof}
	\input{proof_main}
	\subsection{Auxiliary Lemmas}\label{section:auxiliary}
		\input{auxiliary}

\section{Proof of Theorem~\ref{theorem:lreg extension klr}}\label{section:klr proof}
 \input{partI}

%\appendix
%\section{Appendix}
%\input{appendix}

%\backmatter

\bibliographystyle{plain}
\bibliography{refs}

\end{document}

%% file: abstract.tex
% !Tex root = online_ramsey.tex
\begin{abstract}
Consider the following one-player game played on an initially empty graph with $n$ vertices. At each stage a randomly selected new edge is added and the player must immediately color the edge with one of $r$ available colors. Her objective is to color as many edges as possible without creating a monochromatic copy of a fixed graph $F$.

We use container and sparse regularity techniques to prove a tight upper bound on the typical duration of this game with an arbitrary, but fixed, number of colors for a family of $2$-balanced graphs. The bound confirms a conjecture of Marciniszyn, Spöhel and Steger  and yields the first tight result for online graph avoidance games with more than two colors.
\end{abstract}

%% file: introduction.tex
% !Tex root = online_ramsey.tex

Consider the following one-player game played on an initially empty graph on $n$ vertices. In every round we insert a new edge chosen uniformly at random among all non-edges of the graph. The player, henceforth called Painter, must immediately color this edge with one of $r$ available colors. Her objective is to avoid a monochromatic copy of some fixed graph $F$ for as long as possible. We refer to this game as the online $F$-avoidance game with $r$ colors.

We call $N_0\left(F,r,n\right)$ a threshold function for the online $F$-avoidance game with $r$ colors if for every $N \ll N_0\left(F,r,n\right)$ there exists a strategy for Painter that survives for $N$ rounds with high probability and if for every $N\gg N_0\left(r,n\right)$ every strategy fails to survive for $N$ rounds with high probability. Note that such a threshold function always exists \cite[Lemma 2.1]{CPC:4654516}.

This game was first studied by Friedgut, Kohayakawa, Rödl, Ruci\'nski and Tetali, who have shown in \cite{Friedgut:2003:RGA:970118.970124} that $N_0\left(K_3, 2,n\right) = n^{4/3}$ is a threshold for the online triangle-avoidance game with two colors. In 2009 Marciniszyn, Spöhel and Steger proved the following lower bound
\begin{theorem}[\cite{CPC:4654516}]
Let $F$ be a graph that is not a forest, and let $r\geq 1$. Then the online $F$-avoidance game with $r$ colors has a threshold $N_0\left(F,r,n\right)$ that satisfies
\[ N_0\left(F,r,n\right)\geq n^{2-1/\overline{m}^r_2\left(F\right)}, \]
where $\overline{m}^2_r\left(F\right)$ is given by
\[
	\overline{m}_2^r\left(F\right) \coloneqq
\begin{cases}
	\max_{H\subseteq F}\frac{e_H}{v_H} & \text{if $r=1$,} \\
	\max_{H\subseteq F}\frac{e_H}{v_H-2+1/\overline{m}^{r-1}_{2}\left(F\right)} & \text{if $r\geq 2$.}
\end{cases}
\]
\end{theorem}

In an accompanying paper they provide matching upper bounds in the two color case for a large class of graphs, which includes cycles and cliques:
\begin{theorem}[\cite{CPC:4654504}]
Let $F$ be a graph that is not a forest which has a subgraph $F_-\subset F$ with $e_F - 1 $ edges satisfying
\[
m_2\left(F_-\right)\leq \overline{m}^2_2\left(F\right).
\]
Then the threshold for the online $F$-avoidance coloring game with two colors is
\[
N_0\left(F,2,n\right)=n^{2-1/\overline{m}^2_2\left(F\right)}.
\]
\end{theorem}
They conjecture that a similar result is true for all $r\geq 3$.

For three or more colors no tight upper bounds are known. The corresponding offline triangle avoidance game (where Painter gets to see all $N$ edges at once) has a threshold given by $N=n^{3/2}$ \cite{rodl1995}. Clearly this upper bound also applies to the online game. In \cite{doi:10.1137/110826308} Belfrage, Mütze and Spöhel connected the probabilistic one player game to a deterministic two player game originally introduced by Kurek and Ruci\'{n}ski in \cite{KurekRucinski05}. In this version of the game the edges are no longer presented in a random order but can be chosen by a second player called Builder. They show that if there exists a winning strategy for Builder which only creates subgraphs of density at most $d$, then $n^{2-1/d}$ is an upper bound for the threshold of the original probabilistic game.

This technique was used by Balogh and Butterfield in \cite{Balogh20103653} to improve the upper bound for the online triangle avoidance game to $n^{3/2 - c_r}$ for some constant $c_r>0$. Thus the thresholds of the online and offline games differ. Still their upper bound of $n^{2/3-c_r}$ does not match the lower bound provided by Marciniszyn, Spöhel and Steger.

Our contribution is an upper bound, which matches the lower bound Marciniszyn, Spöhel and Steger, for an arbitrary number of colors. That is we show the following:
\begin{theorem}\label{theorem:main theorem}
Let $F$ be a $2$-balanced graph that is not a tree which has a subgraph $F_-\subset F$ with $e_F-1$ edges satisfying
\[m_2(F_-)\leq \overline m_2^2(F).\]
Then the threshold for the online $F$-avoidance game with $r$ colors is
\[ N_0\left(F,r,n\right)= n^{2-1/\overline{m}^r_2\left(F\right)}. \]
\end{theorem}
The premise of our theorem is satisfied by a large class of graphs, which includes cycles and cliques. The condition that $F$ is $2$-balanced is used only for technical reasons. On the other hand the second condition is (in general) necessary. In \cite{CPC:4654516} the authors give an example of a graph (two triangles intersecting in a single vertex) for which the above threshold is incorrect.

To go from two to more colors we prove a generalization of the K\L R-conjecture. For two colors the (unmodified) K\L R-conjecture immediately tells us that Painter has to color on the order of $n^{v_{F_-}}p^{e_{F_-}}$ copies of $F_-$ with the majority color (where $p\asymp n^{-1/\overline m^2_2}$). In expectation a $p$-fraction of those copies of $F_-$ will form a copy of $F$ which contains one edge colored in the secondary color. The density of those edges is roughly $n^{v_{F}-2}p^{e_{F}} = n^{-1/m(F)}$ so we may expect them to form a copy of $F$. A.a.s.\ this is indeed the case and thus Painter looses the game after $\om{n^{2-1/\overline m_2^2(F)}}$ edges have been presented.

To generalize this argument to three colors we want to show that there exist copies of $F_-$ in the primary and secondary colors which share their missing (non)-edge. These non-edges, if they appear later, will have to be colored with the tertiary color and as before the $\overline m_2^r$ density is large enough to guarantee that a.a.s.\ Painter will have to close a copy of $F$ in the tertiary color.

To find the aforementioned copies we prove a variant of the K\L R-conjecture, which allows us find copies of $F_-$ where the missing edge lies in some fixed set of (non-)edges. The proof of this statement uses the container theorem introduced by Saxton and Thomason \cite{SaxtonThomason:container} and independently by Balogh, Morris and Samotij \cite{BaloghMorrisSamotij05}.

%% file: preliminaries.tex
% !Tex root = online_ramsey.tex

For $n\in \N$ let $\set{r}=\left\{1,\dots,r\right\}$. For sets $V, V'$  and $\eps\in\ccint{0}{1}$ we write $V'\subseteq_\eps V$ to denote that $V'$ is a subset of $V$ of cardinality at least $\eps \abs V$.  We say that a statement holds asymptotically almost surely (a.a.s.) if it holds with probability $1-\o{1}$.
%We denote with $G=\left(V, E\right)$ the graph on the vertex set $V=V\left(G\right)$ with edge set $E=E\left(G\right)\subseteq \binom{V}{2}$.
The underlying uncolored graph of the game follows the random graph process $\left(G\left(n,N\right)\right)_{1\leq N \leq \binom{n}{2}}$, where the edges are added in an order selected uniformly at random from the $\binom{n}{2}!$ possible permutations. Let $G_{n,p}$ denote the binomial random graph on $n$ vertices where every edge is present with probability $p$ independently of all others. If $N\asymp p\binom{n}{2}$ then the two models are equivalent in terms of asymptotic properties \cite{janson2011random}. We will thus mostly work with $G_{n,p}$.

Let $G$ be a graph and let $R\subseteq V\left(H\right)$ denote an ordered subset of the vertices. We call the pair $\left(R, G\right)$ a rooted graph. We denote the number of vertices of $G$ with $v_G$ and the number of edges with $e_G$. For a rooted graph we set $\overline v_{R,G} = v_G - \abs R$ and $\overline e_{R,G} = e_G - e_{G[R]}$. For convenience we drop the dependence on $R$ if the set of roots is obvious from the context. That is $\overline v_G = \overline v_{R,G}$ and $\overline e_G = \overline e_{R,G}$. We write $H\subseteq_R G$ do denote that $H$ is a subgraph of $G$ with $R\subseteq V\left(H\right)$ and $H[R]=G[R]$. For a rooted graph $(R,F)$ we denote with $F_-$ the subgraph of $F$ obtained by removing all edges of $F[R]$. We write $(e,F)$ to indicate that the set of roots has cardinality two (this notation does not imply that $e\in E(F)$).

For two rooted graphs $(R, G)$, $(e, F)$ we denote with $(R,G)\times (e,F)$ the graph obtained by attaching to every non root edge $e'$ of $(R,G)$ a new copy of $F$ rooted in $e'$ (possibly removing the edge $e'$ if it is not present in $F$). In general one can choose to orient the attached copies of $F$ in two different ways. For our purposes the actual choice does not matter so we fix one based on the lexicographic ordering of the vertices.

For a collection of graphs $(R,G_1),\dots,(R,G_k)$  which agree on $R$ we denote with 
$\bigsqcup^k_{i=1} (R,G_i)$ the rooted graph obtained by joining pairwise disjoint copies of $G_1,\dots,G_k$ together at their roots.

For graphs we define the following densities (by convention $0/0=0$):
\begin{alignat*}{3}
 d\left(G\right) &= \frac{e_G}{v_G}  &
	m\left(G\right) &= \max_{H\subseteq G} d\left(H\right) \\
 d_1\left(G\right) &= \frac{e_G}{v_G-1} &
	m_1\left(G\right) &= \max_{H\subseteq G} d_1\left(H\right) \\
 d_2\left(G\right) &= \frac{e_G-1}{v_G-2} &
	m_2\left(G\right) &= \max_{H\subseteq G} d_2\left(H\right) \\
 \overline d^r_2\left(G, H\right) &= \begin{cases}
								\frac{e_H}{v_H} & \text{if } r = 1, \\
								\frac{e_H}{v_H-2+1/\overline m_2^{r-1}\left(G\right)} & \text {if } r \geq 2.
							\end{cases}&\quad
	\overline m^r_2\left(G\right) &= \max_{\substack{H\subseteq G}} \overline d^r_2\left(G,H\right)
\end{alignat*}
We say that a graph $G$ is (strictly) balanced with respect to a density function if the maximum is attained (uniquely) by $G$. We say that $G$ is balanced ($1$-balanced, $2$-balanced) if it is balanced with respect to $m$ ($m_1$, $m_2$).

One can check (see \cite{CPC:4654516}) that for every graph $G$ we have
\[
m(G) = \overline m^1_2(G) < \overline m^2_2(G)  < \dots < \overline m^r_2(G) < \dots <  m_2(G).
\]
Furthermore if $G$ is $2$-balanced then it is also balanced with respect to $\overline m^r_2$ for all $r$. It is also easy to check that for every graph $G$ which is not a forest
\[
m_1(G) \leq \overline m^2_2(G),
\]
and that if $G$ is additionally $2$-balanced then it is also strictly $1$-balanced.

The density of a rooted graph is defined by
\[
 d\left(R,G\right) = \frac{\overline e_G}{\overline v_G}  \quad\quad\quad\quad
m\left(R,G\right) = \max_{H\subseteq G} d\left(R\cap V(H),H\right).
\]
As in the unrooted case we call a rooted graph balanced if it is balanced with respect to $m$.

Assume that there exists $G'\subseteq G_{n,p}$ such that $G'\sim G-G[R]$. We say that $G'$ is a copy of $G-G[R]$ in $G_{n,p}$ and that the vertices of $G'$ which correspond to the roots of $(R,G)$ span a copy of $(R,G)$. Observe that the edges between root vertices are immaterial. We will make heavy use of the following upper bound due to Spencer on the number of rooted graphs spanned by vertices of the random graph.
\begin{theorem}[\cite{Spencer:countingextensions}]\label{theorem:spencer}
Let $(R,G)$ be a rooted graph and suppose that $t>m(R,G)$ and $p(n)=\Om{n^{-1/t}}$.  Then a.a.s.\ in $G_{n,p}$ every $\abs R$-tuple of vertices spans $(1\pm\o{1})\mu$ copies of $(R,G$) where $\mu\asymp n^{\overline v_G} p^{\overline e_G}$ is the expected number of such copies.
\end{theorem}
If $p$ is below the density of $(R,G)$ then the following easy to show upper bound will suffice:
\begin{lemma}\label{lemma:span constant}
Suppose that $(R,G)$ is a balanced rooted graph and that $t<m(R,G)$. Then there exists a constant $D(t)$ such that for $p\leq n^{-1/t}$ with probability $1-\o{1}$ no set of $\abs R$ vertices in $G_{n,p}$ spans more than $D$ copies of $G$.
\end{lemma}
\begin{proof}
Let us first prove that the balancedness of $(R,G)$ implies that $t<m(R',G)$ for all $R'$ with $R\subseteq R'\subsetneq V(G)$. As $(R,G)$ is balanced we have for $q=m^{-1/m(R,G)}$
\[
n^{v_G - \abs{R'}} q^{e_G - e_{G[R']}} = \frac{n^{v_G - \abs{R}} q^{e_G - e_{G[R]}} }{n^{\abs{R'} - \abs{R}} q^{e_{G[R']} - e_{G[R]}}}  = \frac{\The{1}}{\Om{1}} = \O{1},
\]
and thus $m(R',F) \geq m(R,F)>t$.

Now the probability that a fixed set of roots spans $C$ pairwise edge disjoint copies of $(R',G)$ is at most 
\[
(n^{\overline v_{R',G}}p^{\overline e_{R',G}})^C=\left(n^{-\The{1}}\right)^C= \o{n^{-v_G}},
\]
provided that $C$ is large enough depending on $t$ and $v_G$. Using the union bound we conclude that a.a.s.\ for every $R'\supseteq R$ no set of $\abs {R'}$ vertices spans more than $C$ pairwise edge disjoint copies of $(R',G)$.

Fix a set of roots $R\subseteq V(G_{n,p})$ and a maximal set of edge disjoint copies of $(R,G)$ spanned by $R$. Every other copy of $(R,G)$ spanned by $R$ must intersect these copies in some set $R'\supsetneq R$. By induction $R'$ spans at most a constant number of copies of $(R', G)$ and since the number of choices for $R'$ is a constant the total number of copies of $(R,G)$ spanned by $R$ is a constant as well.
\end{proof}

The final density of interest is a generalization of the $2$-density to rooted graphs.
For $t>0$ we define
\[
 d_2\left(R,G,t\right) = \begin{cases}
\frac{\overline e_G-1}{v_G-2-t e_{G[R]}} &  \text{if } v_G-2-t e_{G[R]} > 0, \\
\infty & \text{otherwise.}
\end{cases}
\]
And
\[
	m_2\left(R,G,t\right) = 	\max_{\mathclap{\substack{H\subseteq G \\ e_H - e_{H[R]} > 1}}} d_2\left(R\cap V\left(H\right),H,t\right).
\]

The motivation for this definition is given in \autoref{section:klr statement}.

%% file: regularitylemma.tex
% !Tex root = online_ramsey.tex

Our proof relies heavily on the sparse regularity lemma and related concepts. The required definitions and theorems are briefly stated below. A more in depth introduction to the topic can be found in \cite{Gerke:2005}.

\begin{definition}
	A bipartite graph $B=\left(U\cup W,E\right)$ is called \reg{\epsilon}{p} if for all $U'\subseteq_\eps U$ and $W'\subseteq_\eps W$,
	\begin{displaymath}
		\abs{\frac{\abs{E\left(U',W'\right)}}{\abs{U'}\abs{W'}}-\frac{\abs E}{\abs{U}\abs{W}}}\leq\epsilon p.
	\end{displaymath}
	We write \reg{\epsilon} in case $p$ equals the density $\abs E/\left(\abs{U}\abs{W}\right)$.
\end{definition}

The original regularity lemma of Szemer\'edi allows us to partition arbitrary graphs into a constant number of \reg{\epsilon}{1} pairs. Kohayakawa\cite{Kohayakawa1997216} and Rödl (unpublished) independently introduced an analogue of Szemer\'edi's regularity lemma which gives meaningful results for $p\to0$. The generalization works for a class of graph which do not contain large dense spots.
\begin{definition}
Let $G=\left(V,E\right)$ be a graph and let $0<\eta,p\leq 1$. We say that $G$ is \uf{\eta}{p} if for all disjoint sets $U,W\subseteq_\eta V$
\[\abs{E\left(U,W\right)}\leq (1+\eta)p\abs U \abs W.\]
\end{definition}

We can now state Szemer\'edi's regularity lemma for sparse graphs. We use the second version presented in \cite{Kohayakawa1997216}.
\begin{definition}
A partiton $\left(V_i\right)^k_{0}$ of the vertex set $V$ is called an \reg{\epsilon}{p} partition with exceptional class $V_0$ if $\abs{V_1}=\abs{V_2}=\dots=\abs{V_k}$, $\abs{V_0}\leq\epsilon n$, and, with the exception of at most $\epsilon k^2$ pairs, the pairs $\left(V_i, V_j\right)$, $1\leq i \leq j \leq k$ are \reg{\epsilon}{p}.
\end{definition}
\begin{theorem}[sparse regularity lemma]\label{theorem:sparse regularity lemma} For any $\epsilon>0$ and $m_0\geq 1$, there are constants $\eta=\eta\left(\epsilon,m_0\right)>0$ and $M_0 = M_0\left(\epsilon,m_0\right)\geq m_0$ such that for any $p>0$, any \uf{\eta}{p} graph with at least $m_0$ vertices admits an \reg{\epsilon}{p} partition $\left(V_i\right)^k_{i=0}$ with exceptional class $V_0$ such that $m_0\leq k \leq M_0$.\end{theorem}

The \reg{\epsilon}{p}ity of a pair does not imply any lower bounds on its density. In fact the empty graph is \reg{\epsilon}{p} for all $\epsilon>0$, $0\leq p\leq1$. Still it is not hard to show that if $G$ has density at least $\alpha p$ then, for $\eta, \eps$ small enough, we find at least one pair $V_i,V_j$ which is \reg{\eps}{p} with density at least $\alpha p/2$ (and thus \reg{2\eps/\alpha}).

%% file: intro_klr.tex
% !Tex root = online_ramsey.tex

Fix a graph $F$ and let $\left(V_i\right)_{i\in V(F)}$ denote pairwise disjoint sets of size $n$. We call a graph $G$ on the vertex set $\cup_{i\in V(F)} V_i$ \emph{\reg{F}{\eps}} if for every $\{i,j\}\in E(F)$ the pair $(V_i, V_j)$ is \reg{\eps}.  We denote with $\G{F}{n}{m}{\eps}$ the class of \reg{F}{\eps} graphs $G$ for which for every $i,j\in V(F)$
\[
\abs{E(V_i, V_j)} =
\begin{cases}
	m & \text{if }\{i,j\}\in E(F), \\
	0 & otherwise.
\end{cases}
\]

A \emph{partite copy} of $F$ in $G$ is a set of vertices $\left\{v_i \in V_i\colon i\in V(F)\right\}$ such that $\left\{v_i, v_j\right\}\in E(G)$ whenever $\left\{i,j\right\}\in E(F)$. In \cite{KohayakawaLuczakRodl97} Kohayakawa, \L uczak and R\"{o}dl conjectured that almost all graphs in $\G{F}{n}{m}{\eps}$ contain a partite copy of $F$. This conjecture, known as the K\L R-conjecture, was recently proven in full by Saxton and Thomason \cite{SaxtonThomason:container} and independently by Balogh, Morris and Samotij \cite{BaloghMorrisSamotij05}. The following counting version is due to Saxton and Thomason.

\begin{theorem}[K\L R conjecture, weak counting version \cite{SaxtonThomason:container}]\label{theorem:klr}
Let $F$ be a graph and let $\beta >0$. There exists $\mu(\beta) >0$ such that for $n$ sufficiently large and $m\geq \mu^{-1} n^{2-1/m_2(F)}$ the number of graphs in $\G{F}{n}{m}{\mu}$ which do not contain at least $\mu(m/n^2)^{e_F}n^{v_F}$ partite copies of $F$ is at most
\[\beta^m \binom{n^2}{m}^{e_F}.\]
\end{theorem}

Our main tool will be a slight generalization of this theorem: we want to count only those copies of $F$ which satisfy some additional constraints.  These constraints take the form of a partite hypergraph on a subset of the vertex partitions. For a rooted graph $(R,F)$ we denote with $\mathcal R(R,n)$ the class of partite $\abs R$-uniform hypergraphs on the partitions $V_{1},\dots,V_{\abs R}$. Fix $G_R\in \mathcal R(R,n)$ and $G\in\G{F_-}{n}{m}{\eps}$. We denote with $T(G, G_R)$ the multi-hypergraph on $V_{1},\dots,V_{\abs R}$ which contains an edge $e\in E(G_R)$ with multiplicity $k$ if $G$ contains exactly $k$ partite copies of $F_-$ which contain all vertices from $e$.

For our theorem to work we require that the edges of $G_R$ are roughly distributed like partite copies of $F[R]$ in a random $\abs R$-partite graph. This notion is formalized in the following two definitions.
\begin{definition} We say that $G_R\in\mathcal R(R,n)$ is \lreg{F}{q}{\eps} if all tuples of subsets $V'_1\subseteq_{\eps} V_1,\dots,V'_{\abs R}\subseteq_{\eps} V_{\abs R}$ induce at least $q^{e_{F[R]}}\prod_{i\in R}\abs{V'_i}$ edges.
\end{definition}
\begin{definition} We say that $G_R\in\mathcal R(R,n)$ is \ue{F}{q} if for every induced subgraph $F'\subseteq F[R]$ the degree of all tuples from
$\bigtimes_{i \in V\left(F'\right))} V_i$
is at most
\[q^{e_{F[R]}-e_{F'}} n^{\abs R-v_{F'}}.\]
\end{definition}

With these definitions at hand we can state our generalization of \autoref{theorem:klr}.
\begin{restatable}{theorem}{extensionklr}\label{theorem:lreg extension klr}
Let $(R,F)$ be a rooted graph. For every $\beta>0$, $A\geq 1$ there exists $\alpha(A,\beta), \mu(\beta)>0$ such that for every $q(n)=\o{1}$ the following holds:

For $n$ large enough suppose that $m\geq \alpha^{-1} n^{2-1/m_2\left(R,F,-\log_n q\right)}$ and that $G_R\in\mathcal R(R,n)$ is \ue{F}{Aq}  as well as \lreg{F}{q}{\mu}.
Then the number of graphs $G$ in $\G{F_-}{n}{m}{\mu}$ for which $T(G, G_R)$ contains fewer than
$\alpha (m/n^2)^{e_{F_-}}q^{e_F[R]}n^{v_F}$
edges is at most 
\[\beta^m \binom{n^2}{m}^{e_F-e_{F[R]}}.\]
\end{restatable}
The proof follows the proof of the K\L R conjecture presented in \cite{SaxtonThomason:container} and is deferred to \autoref{section:klr proof}.

%% file: proof_main.tex
% !Tex root = online_ramsey.tex

We will assume that $F$ is a fixed $2$-balanced graph which contains an edge $e\in E(F)$ such that $m_2(F-e)\leq \overline m^2_2(F)$. This fixes a rooted graph $(e,F)$. Based on the choice of $e$ we now define the classes $\mathcal F^1,\mathcal F^2, \dots$ of rooted graphs. $\mathcal F^1$ consists of a singular rooted graph: an edge rooted in its endpoints. For $k\geq 2$ we define 
\[
\mathcal F^k \coloneqq \left\{\bigsqcup_{i<k} (e, F)\times(e_i, F^*_i) \mid \forall i\colon (e_i, F^*_i) \in \mathcal F^{\leq i}\right\},
\]
where $\mathcal F^{\leq i}\coloneqq \bigcup_{j\leq i} \mathcal F^j$.  It is useful to observe that every $F^*\in \mathcal F^k$, $k\geq 2$ can be built by starting with a copy of $F$ and then repeatedly attaching a copy of $(e,F)$ to some edge. Since $F$ is $2$-balanced this implies that $F^*$ is $2$-balanced with the same $2$-density (see \autoref{lemma:2-balanced building}). 

If Painter employs the greedy strategy then all edges colored with her $k$-th favorite color will span a copy of the densest graph from $\mathcal F^k$. We will ultimately show that, for some $F^*\in\mathcal F^{\leq k}$, Painter will have to color a linear fraction of all edges which span a copy of $F^*$ with the $k$-color (up to a permutation of the colors).  To this end let us define the notion of a \emph{dangerous} copy of $F^*_-$.
\begin{definition}
Let $F^*\in\mathcal F^k$. We say that an $r$-coloring of $F^*_-$ is dangerous if the $k-1$ copies of $F_-$ whose roots were identified in the construction of $F^*_-$ are all monochromatic and colored with pairwise different colors. We say that $F^*_-$ is dangerous if it is colored with a dangerous coloring. We say that $F\times (e,F^*_-)$ is dangerous if all attached copies of $F^*_-$ are colored according to the same dangerous coloring.
\end{definition}

%
%These classes are motivated as follows: we will show that up to a permutation of the colors for every $k$ there exists $(e_k, F^*_k)\in \mathcal F_k$ such that Painter has colored a constant fraction of all edges which span a copy of $(e_k,F^*_k)$ with color $k$. Since the number of colors is a constant the statement is trivial for $k=1$. For larger $k$ we proceed by induction. We firstly show that for every $i<k$ there are many copies of $(e,F_-)\times(e_i, F^*_i)$ (where the "inner" copy of $F_-$ is monochromatic in color $i$). Secondly, we show that these copies combine to copies of $\bigsqcup_{i,k} (e,F_-)\times(e_i, F^*_i)$. Finally we observe that the missing "inner" edge $e$ cannot be colored with some color $i<k$. Thus Painter will be forced to create copies of $\bigsqcup_{i<k} (e,F)\times(e_i, F^*_i)\in\mathcal F_k$ where $e$ receives (wlog.) color $k$.
%

Assume that Painter has crated a dangerous copy of $(e,F^*_-)$ where  $F^*\in \mathcal F^k$. If $e$ appears as an edge in a later round then Painter will be forced to color it with one of the remaining $r-k+1$ colors or close a monochromatic copy of $F$. In particular if $F^*\in \mathcal F^r$ and Painter creates a dangerous copy of $F\times(e,F^*_-)$ then Painter cannot color the inner copy of $F$ (if it were to appear) without creating a monochromatic copy of $F$. The following Lemma states that Painter cannot avoid such dangerous copies of $F\times(e,F^*_-)$.

\begin{restatable}{lemma}{dangerous}\label{lemma:dangerous}
Fix a function $p=p(n)$ satisfying $n^{-1/\overline m_2^r(F)} \ll p \ll n^{-1/\overline m^r_2(F)}\log n$. Then there exists constants $c,C>0$ such that a.a.s.\ after $Cn^2p$ rounds there either exists a monochromatic copy of $F$ or we find a graph $(e,F^*)\in\mathcal F^r$ such that Painter has created $cn^{v_F} \left(n^{v_{F^*_- -2}}p^{e_{F^*_-}}\right)^{e_F}$ dangerous copies of $F\times\left(e,F^*_-\right)$.
\end{restatable}
In other words at least a constant fraction of the copies of $F\times (e, F^*_-)$ are dangerous (for some $F^*\in\mathcal F^r$). Assuming the above Lemma the main result follows from a second moment argument similarly to the one presented in \cite{CPC:4654504}. For completeness we restate the proof below. We shall also require the following proposition, whose proof we defer to \autoref{section:auxiliary}.
\begin{restatable}{proposition}{FtimesFstardensity}\label{prop:FtimesFstar density}
All rooted graphs $(e,F^*)\in \mathcal F^r$ satisfy 
\[
m(F\times\left(e,F^*\right)) \leq \overline m^r_2(F).
\]
\end{restatable}
\begin{proof}[Proof of main theorem]
We pause the game after $m=\The{n^2p}$ rounds. Exploiting the asymptotic equivalence between $G_{n,m}$ and $G_{n,p}$ we consider the resulting graph to be distributed like a $G_{n,p}$.

For a graph $G$ let the random variable $X_G$ denote the number of copies of $G$ in $G_{n,p}$. Let $F^*$ denote the graph guaranteed by \autoref{lemma:dangerous} and define $\tilde F\coloneqq F\times (e,F^*)$ and $\tilde F_-\coloneqq F\times (e,F^*_-)$. By \autoref{lemma:dangerous} Painter has created $M=\Om{X_{\tilde F_-}}$ dangerous copies of $\tilde F_-$. We now consider these $M$ copies to be fixed and for $i\in \set{M}$ denote with $F_i$ the missing inner copy of $F$ of the $i$-th copy of $\tilde F_-$. Observe that $F_i$ and $F_j$ are not required to be disjoint and may in fact be identical.

Observe that if Painter is forced to color one of the $F_i$ in a future round then she must close a monochromatic copy of $F$ and thus loose the game. We now show that indeed a.a.s.\ one of the $F_i$ appears within the next $\The{n^2p}$ rounds.

Let $Z_i$ denote the event that $F_i$ appears and let $Z=\sum_{i=1}^M Z_i$. We have
\[
	\E{Z}
	= M p^{e_F}
	=\Om{\E{X_{\tilde F_-}}}p^{e_F}
	= \Om{\E{X_{\tilde F}}}
	\overset{(*)}{=} \om{1},
\]
where $(*)$ follows from \autoref{prop:FtimesFstar density}. Furthermore
\begin{align*}
\Var{Z} &= \E{Z^2}-\E{Z}^2 \\
&= \sum_{i,j} \E{Z_i Z_j} - \E{Z_i}\E{Z_j}\\
&= \sum_{\substack{G\subseteq F \\ e_G \geq 1}} \sum_{\substack{i,j\\ F_i\cap F_j \sim G}} p^{2e_F-e_G} - p^{2e_F} \\
&\leq \sum_{\substack{G\subseteq F \\ e_G \geq 1}} M_G p^{2e_F-e_G},
\end{align*}
where $M_G$ denotes the number of pairs of $F^*_-$ whose (missing) inner copies of $F$ intersect in a copy of $G$.

Fix $G\subseteq F$ and let $H$ denote a graph obtained as the union of two copies of $\tilde F_-$ whose (missing) inner copies intersect in a copy of $G$. Let $T$ denote their intersection and $T_+$ the graph obtained by adding the missing edges of the inner copy of $G$ to $T$. Observe that $T_+\subseteq \tilde F$ and thus by \autoref{prop:FtimesFstar density} $\E{X_{T_+}}=\om{1}$.

We have
\[
\E{X_H} = \The{\frac{\E{X_{\tilde F_-}}^2}{\E{X_{T}}}} = \The{\frac{\E{X_{\tilde F_-}}^2p^{e_G}}{\E{X_{T_+}}}} = \o{\E{X_{\tilde F_-}}^2p^{e_G}}.
\]
Since for every $G\subseteq F$ the number of choices for $H$ is constant we have (over the first $\The{n^2p}$ rounds) $\E{M_G}=\o{\E{X_{\tilde F_-}}^2p^{e_G}}$ and thus by first moment method $M_G=\o{\E{X_{\tilde F_-}}^2p^{e_G}}$ a.a.s.

This implies (over the second set of $\The{n^2p}$ rounds) $\Var{Z}=\o{\E{Z}^2}$ and thus $Z\geq 1$ a.a.s.
\end{proof}

\subsection{Proof of Lemma~\ref{lemma:dangerous} }
Fix a function $p=p(n)$ which satisfies $n^{-1/\overline m^r_2\left(F\right)} \ll p \ll n^{-1/\overline m^r_2\left(F\right)}\log n$.
We will divide the game into a constant number of phases. In each phase we sample a copy of the binomial random graph $G_{n,p}$ and present its edges to Painter in random order (edges already presented in a previous phase are ignored). A.a.s.\ in each phase at most $\The{n^2p}$ edges are presented. Denote with $G^k_{n,p}$ the colored graph after $k$ phases. We implicitly assume that $G^k_{n,p}$ does not contain a monochromatic copy of $F$.

As a main step in the proof we will show that for every set $S$ of at most $r-2$ colors Painter must create a graph $G\in\G{K_t}{\tilde n}{m}{\eps}$ monochromatic in some color from $\set{r}\setminus S$ after a constant number of phases. In general we cannot expect that $m=\Om{n^2p}$ (for example a greedy Painter will only produce a single color class with this density). Instead we will require that $m=\The{n^{v_{F^*}}p^{e_{F^*}}}$ for some $F^*\in\mathcal F^{\abs {S} + 1}$. To retain some control on these graphs we introduce the concept of an $F^*$-spanning subgraph.
\begin{definition}\label{def:spanning}For a rooted graph $(e,F^*)\in \mathcal F^k$ we say that a subgraph $G\subseteq G^k_{n,p}$ is $F^*$-spanning if for every edge $e\in E(G)$ there exists $F^*(e)\sim F^*$ in $G^k_{n,p}$ such that
\begin{enumerate}
\item the endpoints of $e$ are the roots of $(e,F^*(e))$,
\item all non root vertices of $F^*(e)$ lie outside of $V(G)$ and,
\item $F^*(e)$ and $F^*(e')$ are edge disjoint for all $e'\in E(G)\setminus\left\{e\right\}$.
\end{enumerate}
\end{definition}

We shall see later that an $F^*$-spanning subgraph behaves like a $G_{n,q}$ with $q=n^{v_{F^*}-2}p^{e_{F^*}}$ in the sense that we obtain  bounds on its maximum degree as well as exponential upper bounds on the number of edges between linear sized vertex sets.

We are now in a position to state the main Lemma of this subsection.
\begin{restatable}{lemma}{forceregular}\label{lemma:force regular}
Fix a set $S$ of at most $r-2$ colors and an integer $t\geq2$. Then there exist a positive integer $k$ and a constant $\delta>0$ such that for every $\eps>0$ there exists $\eta>0$ such that for $p=\om{n^{-1/\overline m^r_2(F)}}$ a.a.s.\ in $G^k_{n,p}$  we find a subgraph $G\in \G{K_t}{\tilde n}{m}{\eps}$ which is monochromatic in some color from $\set{r}\setminus S$ and $F^*$-spanning in $G^k_{n,p}$ where $F^*\in\mathcal F^{\leq \abs {S}  + 1}$, $m\geq \eta n^{v_{F^*}}p^{e_{F^*}}$ and  $\tilde n\geq \eta n$.

Furthermore for every choice of $\tilde n$, $m$ and graphs $F^*\in \mathcal F^{\leq \abs {S}  + 1}$ and $G'$ with $\abs{E(G')}=\om{n}$ the probability that the statement nominates $\tilde n$, $m$, $F^*$ and $G\supseteq G'$ is at most
\[
\left(\frac{m}{\tilde n^2 \delta}\right)^{\abs{E(G')}}.
\]
\end{restatable}
It is crucial that in the probability bound we loose only a constant factor (the $\delta$) independently of the requested regularity (as opposed to the density of $G$ which depends on $\eta(\eps)$).

The next lemma states that this is the density guaranteed by \autoref{lemma:force regular} has the right order of magnitude in the sense that if we forbid $\abs {S}$ colors then the resulting graph should have high enough density for Painter to loose the game with $r-\abs{S}$ colors.

\begin{lemma}\label{lemma:f star density}Every $F^*\in\mathcal F^{k}$ satisfies
\[
	n^{v_{F^*}-2}n^{-e_{F^*}/\overline m^r_2(F)}
	\geq n^{-1/\overline m^{r-k+1}_2(F)},
\]
provided that $k\leq r$.
\end{lemma}
\begin{proof}
The proof proceeds by induction on $k$. The singular graph in $\mathcal F^{1}$ consists of a single edge and thus the statement holds for $k=1$.

For $k\geq 2$ let $F^*_1,\dots,F^*_{k-1}$ denote the graphs used during the construction of $F^*\in\mathcal F^k$. Writing $p_i = n^{-1/\overline m^{i}_2(F)}$ we have
\begin{equation}
	n^{v_{F^*}-2}p_r^{e_{F^*}}
	= p_r \prod_{i<k} n^{v_F-2} \left(n^{v_{F^*_i}-2}p_r^{e_{F^*_i}}\right)^{e_F-1}
	\overset{(*)}{\geq} p_r \prod_{i<k} n^{v_F-2} p_{r-i+1}^{e_F-1},
\label{eq:p_i recursion}
\end{equation}
where (*) follows from the induction hypothesis.

By definition of $\overline m^i_2$ we have
\[
n^{v_F-2}p^{e_F}_{r-i+1} \geq p_{r-i}
\]
and thus (\ref{eq:p_i recursion}) is at least
\[
	p_r\prod_{i<k} p_{r-i}/p_{r-i+1}	= p_{r-k+1}.
\]
\end{proof}

We will give a detailed proof of \autoref{lemma:force regular} below. Before that we will walk through the main argument and state a number of auxiliary lemmas.

Assume that (by induction) we have found graphs $G_1,G_2,\dots,G_k$ with
\[
	G_i\in\G{F_-}{\tilde n}{\The{n^{-1/\overline m_2^{r-i+1}(F)}}}{\eps},
\]
which are monochromatic in pairwise different colors. Assume furthermore that $G_i$ is $F^i$-spanning for some $F^i\in\mathcal F^i$ and that the partitions $V_a, V_b$ corresponding to the missing edge of $F_-=F-\left\{\left\{a,b\right\}\right\}$ are the same for all $G_i$. Through repeated application of \autoref{theorem:lreg extension klr} we will be able to count the number of copies of $\bigsqcup_{i\leq k} (e,F_-)$ in $\bigcup_{i\leq k} G_i$ (where the $i$-th copy of $F_-$ is to be from $G_i$).

We expect to find roughly 
\begin{equation}\label{eq:VaVb edges}
n^2 \prod_{i\leq k} n^{v_F-2} n^{-(e_F-1)/\overline m_2^{r-i+1}(F)} = n^2 \prod_{i\leq k} \frac{n^{-1/\overline m_2^{r-i}(F)} }{n^{-1/\overline m_2^{r-i+1}(F)} } = n^{2-1/\overline m_2^{r-k}(F) + 1/\overline m_2^{r}(F)}
\end{equation}
such graphs. The $2$-density of $F$ is strictly above $\overline m^r_2(F)$. Since $F$ is $2$-balanced we have $m_2(F)=m(e,F)=m(e,F_-)$ and  \autoref{lemma:span constant} implies that every pair from $V_a\times V_b$ spans at most a constant number of copies of $(e,F_-)$. Thus the number of pairs in $V_a\times V_b$ which span a copy of $F_-$ in each of the graphs $G_i$ is of the same order of magnitude as (\ref{eq:VaVb edges}).

Out of these pairs roughly $n^{2-1/\overline m^{r-k}(F)}$ will appear as actual edges if we present Painter with another set of $n^2p$ edges. If Painter wants to avoid a monochromatic copy of $F$ then she is forced to color these edges with colors distinct from those used in $G_1,\dots,G_k$. Furthermore since all the $G_i$ were $F^i$-spanning these edges all span a copy of 
\[
\bigsqcup_{i\leq k} (e,F)\times (e,F^i)=F^*\in\mathcal F^{k+1}.
\]
We are below the $2$-density of $F^*$ (which equals that $F$) and therefore the following Lemma tells us that this edge set can be turned into an $F^*$-spanning subgraph by discarding a negligible number of edges.
\begin{lemma}\label{lemma:2-balanced intersection}
Suppose that $F$ is a $2$-balanced graph and that $F_1,F_2\sim F$ intersect in at least one, but not all edges. Then $n^{v_F}p^{e_F}\gg n^{v_{F_1\cup F_2}}p^{e_{F_1\cup F_2}}$ provided that $p=\o{n^{-1/m_2(F)}}$.
\end{lemma}
\begin{proof}
For a graph $H$ write $X_H=n^{v_H}p^{e_H}$.
Let $G=F_1\cap F_2$. Since $F$ is $2$-balanced and $v_F,v_G\geq 2 $ we have
\begin{align*}
\frac{X_F}{n^2p}=\left(\frac{p}{n^{-1/m_2(F)}}\right)^{e_F-1} \quad\quad\text{and} \quad\quad \frac{X_G}{n^2p}\geq \left(\frac{p}{n^{-1/m_2(F)}}\right)^{e_G-1}.
\end{align*}
For $p=\o{n^{-1/m_2(F)}}$ we obtain
\[
\frac{X_F}{X_{F_1\cup F_2}} = \frac{X_F}{\frac{X_F^2}{X_G}}  = \frac{X_G}{X_F} \geq \left(\frac{p}{n^{-1/m_2(F)}}\right)^{e_G-e_F} = \om{1},
\]
as desired.
\end{proof}

Finally we will want to apply the sparse regularity lemma to this $F^*$-spanning subgraph. For this we need it to be upper-uniform, which is confirmed in the following lemma.
\begin{lemma}\label{lemma:upper-uniform}Suppose that $p= \om{ n^{-1/\overline m_2^r(F)}}$.
Let $(e,F^*)\in\mathcal F^{\leq r-1}$. Then for every $\eta>0$ a.a.s.\ every $F^*$-spanning subgraph $G$ of $G(n,p)$ with at least $\eta n$ vertices is \uf{\eta}{n^{v_{F^*}-2}p^{e_{F^*}}}.
\end{lemma}
\begin{proof}
The lemma follows from the following extension of the standard Chernoff bound:
\begin{theorem}[\cite{doi:10.1137/S0097539793250767}]\label{theorem:extended chernoff}
Let $X_1,\dots, X_n$ be a sequence of not necessarily independent Bernoulli-distributed random variables which satisfy $\Pr{\bigwedge_{i\in S} X_i}\leq q^{\abs S}$ for all subsets $S\subseteq \set n$. Then for $0<\eps\leq 1$
\[
	\Pr{\sum_{i=1}^n X_i \geq\left(1+\eps\right)qn} \leq e^{-nq \eps^2/3}.
\]
\end{theorem}

For fixed vertex sets $V_1, V_2\subseteq_{\eta^2} V$  let $G\subseteq G(G_{n,p})$ denote a (canonical) $(e,F^*)$-spanning graph in $G_{n,p}$ which maximizes the number of edges between $V_1$ and $V_2$. For $e\in E\left(K_n\left[V_1, V_2\right]\right)$ let $X_e$ denote the indicator random variable for the event $e\in G$.  We have for every set $S$
\[
\Pr{\bigwedge_{e\in S} X_e } \leq \Pr{S \text{ is $(e,F^*)$-spanning in $G_{n,p}$} } \leq \left(n^{v_{F^*}-2}p^{e_{F^*}}\right)^{\abs S}.
\]
And thus by \autoref{theorem:extended chernoff}
\[
\Pr{\abs{E_G(V_1,V_2)} \geq (1+\eta) \abs{V_1}\abs{V_2} n^{v_{F^*}-2}p^{e_{F^*}}} \leq e^{-\The{ n^{v_{F^*}}p^{e_{F^*}}}}.
\]
Since \[
n^{v_{F^*}}p^{e_{F^*}} \overset{\autoref{lemma:f star density}}{\gg} n^{2-1/\overline m_2^{1}(F)} \geq n
\]
a union bound over at most $4^{n}$ choices for $V_1$ and $V_2$ proves the Lemma.
\end{proof}

We thus obtain an $F^*$-spanning graph $G\in\G{K_2}{\tilde n}{\The{n^{v_{F^*}}p^{e_{F^*}}}}{\eps}$. Repeating the argument a constant number of times (by exposing more edges inside one of the two partitions of $G$) one can obtain a monochromatic graph $G_{k+1}\in\G{K_t}{\tilde n}{\The{n^{v_{F^*}}p^{e_{F^*}}}}{\eps}$ as required to finish the induction.

This argument can be repeated as long as $\abs S \leq r-2$. One could hope to iterate one more time and find a graph $G_r\in \G{F_-}{\tilde n}{\The{n^{2-1/\overline m^1_2(F)}}}{\eps}$. This approach is bound to fail. The density of $G_r$ is (in general) not above the $2$-density of $F_-$ so we cannot hope to find copies of $F_-$ in $G_r$. Instead we find graphs $G_1,\dots,G_{r-1}$, where $G_i\in\G{F\times(e,F_-)}{\tilde n}{\The{n^{2-1/\overline m^{r-i+1}(F)}}}{\eps}$, whose \emph{inner} partitions agree and use \autoref{theorem:lreg extension klr} to show directly that many $v_F$-tuples span copies of $F\times(e,F_-)$ in all the $G_i$.

Before formalizing the above we need two more auxiliary lemmas. The first one asserts that the density of $G_i$ is indeed large enough to apply \autoref{theorem:lreg extension klr}.
\begin{restatable}{lemma}{klrdensity}\label{lemma:klr density}
Suppose that $F$ is a $2$-balanced graph, which contains an edge $e$ such that $m_2\left(F-\left\{e\right\}\right)\leq \overline m^2_2(F)$.
Then for all $r\geq k\geq 2$
\begin{align*}
\overline m^{k}_2(F) &\geq m_2(e,F, +1/\overline m_2^{k}(F) - 1/\overline m_2^{r}(F)), \\
\overline m^k_2(F) &\geq m_2(V(F),F\times (e,F),+1/\overline m_2^{k}(F) - 1/\overline m_2^{r}(F)).
\end{align*}
\end{restatable}
Secondly \autoref{theorem:lreg extension klr} requires the (hyper)-graph to be upper-extensible. For us this hypergraph will consist of all pairs (all $v_F$-tuples) which already span a copy of $F_-$ (of $F\times (e,F_-)$) in all graphs $G_1,\dots,G_i$. By \autoref{theorem:spencer} it suffices to show that we are above the rooted density of the corresponding graphs:
\begin{restatable}{lemma}{rooteddensity}\label{lemma:rooted density}
 Let $F^*\in\mathcal F^k$ where $k\geq 2$. Then
\[
m_1(F^*_-) < \overline m^k_2(F)
\]
and for every $V_0\subsetneq V\left(F\right)$ 
\[
m\left(V_0, \left(V_0,F\right)\times\left(e,F^*_-\right)\right) < \overline m^{k}_2(F).
\]
\end{restatable}

We can now state the proof of \autoref{lemma:force regular}.
\forceregular*
\begin{proof}
The proof follows by induction on $\abs S$ and $t$. For $t=2$ and $S=\emptyset$ we apply the sparse regularity lemma (\autoref{theorem:sparse regularity lemma}) to the majority color class (in that case $F^*=(e,e)$ is just an edge -  the unique graph in $\mathcal F^1$).
\begin{description}
\item[$t$ step:] Fix $t>2$ and a set $S$. We will apply the induction hypothesis (for $t\leftarrow2$ and $S\leftarrow S$) $K=r t \abs*{\mathcal F^{\leq \abs S+1}}$ times. Let $k',\delta'$ denote the absolute constants guaranteed for $t\leftarrow 2$. Denote with $\epsilon_i$ the value which we will use for $\epsilon$ in the $i$-th application of the induction hypothesis and let $\eta_i(\epsilon_i)$ denote the guaranteed constant. Our choice for $\epsilon_i$ will depend only on the constants $\eps_j$, $\eta_j$ where $j>i$ and on the requested $\epsilon$.

Apply the induction hypothesis once to $G^{k'}_{n,p}$ for $t\leftarrow2$ and obtain an \reg{\eps_1} graph $G_1\in \G{K_2}{\tilde n_1}{m_1}{\eps_1}$. Let $V_1\subset V(G)$ denote one of its vertex partitions. We then ask painter to color another $G^{k'}_{n,p}$ but we look only at the subgraph induced by $V_1$ (which is distributed like a $G^{k'}_{\abs {V_1},p}$). Since $\abs{V_1}\geq \eta_1 n$ we have $p=\om{\abs{V_1}^{-1/\overline m^r_2(F)}}$ and thus we can apply the induction hypothesis a second time to obtain an \reg{\epsilon_2} graph $G_2$ whose edges are fully contained in $V_1$. We repeat this procedure $K$ times and obtain a sequence of sets $V_1\supset V_2\supset \dots \supset V_K$ and nested graphs $G_{1},\dots,G_{K}$, where $G_i\in \G{K_2}{\tilde n_i}{m_i}{\eps_i}$. Every such $G_i$ nominates a color and a graph $F^*_i\in \mathcal F^{\leq \abs S+1}$ out of at most $r\abs*{F_{\leq \abs S + 1}}$ choices. By the pigeonhole principle we may thus fix a subset $T\subseteq \set{K}$ of size $t$ such that all graphs $G_{i}$ with $i\in T$ nominate the same color and the same graph $F^*\in\mathcal F^{\leq \abs S + 1}$.

Let $\tilde n\coloneqq \tilde n_K$ denote the size of the vertex partitions of $G_K$. We arbitrarily pick sets $\overline{V_{i}}$ of size $\tilde n$ such that
\begin{align*}
\overline{V_1}\subseteq& V(G_1)\setminus V_1, \\ 
\vdots& \\
\overline{V_{K}}\subseteq& V(G_K)\setminus V_{K}.
\end{align*}
Finally set $\overline{V_{K+1}}=V_K$. These sets are pairwise disjoint and for every pair $i<j$ we have $\overline {V_j}\subseteq V_{j-1}\subseteq V_i$. Therefore the sets $\overline{V_i}, \overline {V_j}$ are subsets of the two partitions of $G_i$ and for $\eps_i$ small enough, depending on $\eps$, $\eta_{i+1},\dots,\eta_{K}$, the induced bipartite graph $G_i\left[\overline V_i, \overline V_j\right]$ is \reg{\eps/2} with at least half the density. 
Let \[
m=\min_{\substack{i,j\in  T \\ i<j}} \abs{E\left(G_i\left[\overline V_i, \overline V_j\right]\right)}
\]
and pick for every $i,j\in T$, $i<j$ a subgraph $G_{i,j}\subset G_i\left[\overline V_i, \overline V_j\right]$ with exactly $m$ edges u.a.r.\ among all subgraphs with $m$ edges.
By \autoref{lemma:f star density} we have
\[ 
m= \Om{n^{v_{F^*}}p^{e_{F^*}}} \gg n^{2-1/\overline m_2^{r-\abs S+1}(F)} \geq n
\]
and thus these graphs $G_{i,j}$ will be \reg{\eps} with high probability. Since
\[
\tilde n \geq n \prod_{i\in\set{K}}\eta_i \quad\text{ and } \quad
 m \geq \eta_k \tilde n^{v_F^*}p^{e_F^*}
\] we may set
\[
	G\coloneqq \bigcup_{\substack{i,j\in T \\ i < j}} G_{i,j}\in \G{K_t}{\tilde n}{m}{\eps}.
\]

Furthermore we claim that the graphs $F^*(e)$ guaranteed by the invocations of the induction hypothesis are pairwise edge disjoint. This is because for $e\in E(G_i)$ the graph $F^*(e)$ has no edges inside $V_i$, but for $j>i$ the graphs $F^*(e')$, $e'\in E(G_j)$ lie completely inside $V_i$. Thus $G$ is $F^*$-spanning.

Finally we have to calculate the probability that $G'\subseteq G$ for some graph $G'$ with $\om{n}$ edges. To do so fix $\tilde n$, $m$ and the sets $\overline V_i$ among $2^{\The{n}}$ possibilities. We may assume that all edges of $G'$ go between the sets $\overline V_i$. Since $G$ is the (disjoint) union of random subgraphs of $G_i[\overline V_i, \overline V_j]$ and $\abs{E(G_i[\overline V_i, \overline V_j])} \geq \tilde n^2 (m_i/2\tilde n^2_i)$whose density is at least half as large as the density of $G_i$ the probability that $G'\subseteq G$ is then at most
\begin{align*}
& \prod_{i<j\in T} \Pr{G'[\overline V_i, \overline V_j] \subseteq G_i[\overline V_i, \overline V_j]} \left(\frac{m}{\abs{E(G_i[\overline V_i, \overline V_j])}}\right)^{\abs{E(G'[\overline V_i, \overline V_j] )}} \\
&\leq \prod_{i<j\in T}\left( \frac{m_i}{\tilde n_i^2 \delta'}    \frac{m}{\abs{E(G_i[\overline V_i, \overline V_j])}}\right)^{\abs{E(G'[\overline V_i, \overline V_j] )}} \\ 
&\leq \prod_{i<j\in T}\left( \frac{2m}{\delta' \tilde n^2}\right)^{\abs{E(G'[\overline V_i, \overline V_j] )}} = \left( \frac{2m}{\delta' \tilde n^2}\right)^{\abs{E(G')}}.
\end{align*}

Allowing some room for a union bound over the choices for $\tilde n$, $m$ and the sets $\overline V_i$ we may fix $\delta=\delta'/3$, $\eta =\eta_K\prod_{i\in\set{K}}\eta_i^{v_{F^*}} $ and $k=Kk'$.

\item[$S$ step:] Fix a nonempty set $S$ of at most $r-2$ colors and assume that the statement holds for all sets containing fewer than $\abs S$ colors. Our goal is to show that then the statement holds for $S$ and $t=2$.

Similarly to what we did in the induction step for $t$ we apply the induction hypothesis $\abs{S}$ times in a nested fashion. As before let $k'$, $\delta'$ denote absolute constants for which the induction hypothesis holds for $t\leftarrow v_F-1$ and all subsets of $S$. Denote with $\eps_i$ the value which we will use for $\eps$ in the $i$-th application of the induction hypothesis and let $\eta_i(\eps_i)$ denote the guaranteed constant. Again $\eps_i$ will depend only on $\eps_j$, $\eta_j$, where $j>i$. Crucially $\eps_i$ will not depend on the requested $\eps$ and $\eps_{\abs S}$ will be an absolute constant. Finally for the $i$-th invocation we will pick $S$ as the set of colors of $G_{1},\dots,G_{i-1}$ (thus $S\leftarrow\emptyset$ for $i=1$).

As before we obtain monochromatic graphs $G_1,\dots,G_{\abs S}$, such that $G_i\in \G{K_{v_F-1}}{\tilde n_i}{m_i}{\eps_i}$ and $V(G_i)\subseteq V_{i-1}$ for $V_0=V$ and where $V_i$ is an arbitrary partition of $G_i$.

Assume that one of the $G_i$ is monochromatic in a color from $\set{r}\setminus S$. The density of $G_i$ is in $\The{n^{v_{F^*_i}-2}p^{e_{F^*_i}}}$, where the constant does not depend on $\eps$ (since $\eps_1,\dots,\eps_{\abs S}$ do not depend on $\eps$). Furthermore by \autoref{lemma:upper-uniform} $G_i$ is \uf{\o{1}}{n^{v_{F^*_i}-2}\left(\abs Sk'p\right)^{e_{F^*_i}}}. Thus we can apply the sparse regularity lemma (\autoref{theorem:sparse regularity lemma}) to $G_i$ and obtain a graph from $\G{K_2}{\tilde n}{m}{\eps}$ whose density is of the same order as the density of $G_i$ and we are done.

Otherwise all of the $G_i$ are monochromatic in distinct colors of $S$. We want to show that in $V_{\abs S}$ there are many pairs of vertices which span a copy of $(e,F_-)$ in each of the $G_i$. To this end we define the auxiliary directed graphs $A_i$. Let $A_0$ denote the complete directed graph on $V$. For $i=1,\dots,\abs S$ the vertex set of $A_i$ is $V_i$ and we connect two vertices $x\neq y\in V_i$ if $\left(x,y\right)\in E(A_{i-1})$ and if $(x,y)$ span a partite copy of $F_-$ in $G_{i}$ (partite with respect to the non root vertices, $x$ and $y$ lie in the same partition).

Define $(e,F^*)\coloneqq\bigsqcup_{j\in \set{\abs S}} (e, F)\times(e,F^*_i)\in\mathcal F^{\abs S+1}$. By definition of $A_{\abs S}$ every edge $e\in E\left(A_{\abs S}\right)$ spans a copy of $(e,F_-)$ in each of the $G_i$. Furthermore since every edge of $G_i$ spans a copy of $(e,F^*_i)$ the edge $e$ spans a copy of $(e,F^*_-)$. Finally observe that since the $G_i$ are monochromatic in pairwise different colors of $S$ the edge $e$ (if it would be presented to Painter in some later round) has to be colored with some color from $\set{r}\setminus S$.

We will need the following auxiliary claim about the density of $A_{\abs S}$ whose proof we defer.
\begin{claim}\label{claim:force regular claim}
For every integer $i\leq \abs S$ and $\kappa>0$ and small enough $\eps_1,\dots,\eps_i$ there exists $\gamma(\kappa,\eps_1,\dots,\eps_i)>0$ such that a.a.s.\ for all disjoint and equi-sized subsets $X,Y\subseteq_\kappa V_i$ the induced bipartite subgraph  $A_i[X,Y]$ contains at least $\gamma \abs X \abs Y \prod^i_{j=1} p_j$ edges, where
\[
	p_j = n^{v_{F}-2}\left(n^{v_{F^*_j}-2}p^{e_{F^*_j}}\right)^{e_{F}-1}.
\]
\end{claim}
We invoke the claim for $i\leftarrow\abs S$ and $\kappa\leftarrow1/4$ to lower bound the number of pairs $x,y\in V_{\abs S}$ which span a dangerous copy of $F^*$ by
\[
	\frac{\gamma}{2}\left(\frac{\abs{V_{\abs S}}}{2}\right)^2 \prod^{\abs S}_{j=1} p_i \geq \gamma' n^{v_{F^*}}p^{e_{F^*}-1} \overset{\autoref{lemma:f star density}}{\gg} \frac{n^{2-1/\overline m^{r-\abs S}_2(F)}}{p} \gg \frac{n}{p},
\]
where $\gamma$ is the constant guaranteed by the claim and $\gamma'=\gamma \left(\prod_i \eta_i\right)^2/8$ is an absolute constant, which in particular does not depend on $\eps$.

We then present another $G_{n,p}$ to Painter. Painter will be forced to color at least a $p/r$-fraction of the edges in $A_{\abs S}$ with some color from $\set{r}\setminus S$ (or create a monochromatic copy of $F$). Thus we obtain a monochromatic set of $\gamma' n^{v_{F^*}}p^{e_{F^*}}/r$ edges which all span a copy of $F^*$. Next we remove all edges whose copies of $F^*$ intersect. \autoref{lemma:2-balanced intersection} together with Markov's inequality implies that with probability $1-\o{1}$ we remove only $\o{n^{v_{F^*}}p^{e_{F^*}}}$ edges.

So we are left with a $(e,F^*)$-spanning set of at least $\gamma' n^{v_{F^*}}p^{e_{F^*}}/2 \gg n$ edges $E'$. By \autoref{lemma:upper-uniform} $E'$ is \uf{\o{1}}{n^{v_{F^*}-2}\left(\left(\abs S k'+1\right)p\right)^{e_{F^*}}}. Therefore we may apply the sparse regularity lemma to $E'$ and obtain a graph from $\G{K_2}{\tilde n}{m}{\eps}$ whose density is a constant fraction of the density of $E'$.

Finally the probability that a fixed set of $s$ edges is $(e,F^*)$ spanning is at most $\left(n^{v_{F^*}-2}\left(\left(\abs S k'+1\right)p\right)^{e_{F^*}}\right)^s$. Since $m/\tilde n^2 = \Om{n^{v_{F^*}-2}p^{e_{F^*}}}$ (not depending on $\eps$) the probability bound holds for some $\delta$.

It remains to prove \autoref{claim:force regular claim}. We proceed by induction on $i$. $A_0$ is complete and thus the base case $i=0$ holds vacuously. So let $i\geq 1$ and fix some $\kappa>0$. Denote the lower bound on the density of $A_{i-1}$ guaranteed by the induction by
\[
	q=\Om{\prod_{j<i} p_j} \gg \frac{n^{-1/\overline m^{r-i+1}_2(F)}}{p}.
\] 
We define $\beta=(\delta'/(3e))^{e_{F_-}}$ and 
\[
A=\frac{ \prod^{i-1}_{j=1} p_j}{ q\kappa \prod^i_{j=1} \eta_j} = \The{1}.\]
Let $\mu(\beta)$ and $\alpha(A,\beta)$ denote the constant guaranteed by \autoref{theorem:lreg extension klr} when invoked with $(R,F)\leftarrow (e,F)$, $\beta\leftarrow \beta$ and $A\leftarrow A$.  Fix disjoint equi-sized sets $X, Y\subseteq_\kappa V_i$. 
Write $m=\ceil{\abs X \abs Y m_i/(2\tilde n_i^2)}$ and pick any subgraph $G'_i\subseteq G_i$ from $\G{F_-}{\abs X}{m}{\mu}$ such that its partitions which correspond to the roots of $F_-$ are $X$ and $Y$ (taking suitable vertex sets and a random subset of $m$ edges from each partition succeeds with probability $1-2^{-\The{m}}$).

If $T(G', A_{i-1}[X,Y])$ contains at least 
\[
\alpha \left(\frac{m}{\abs X^2}\right)^{e_{F_-}} \abs X^{v_F} q= \The{n^2\prod_{j\leq i} p_j}
\]
 edges, then since by \autoref{lemma:span constant} every edge $e$ spans at most a constant number of copies of $F_-$ the density of $A_{i}[X,Y]$ is of the correct order of magnitude.

Otherwise we want to apply \autoref{theorem:lreg extension klr} with $G\leftarrow G'$ and $G_R \leftarrow A_{i-1}[X,Y]$ (viewed as an undirected graph) and $q\leftarrow q$. To apply \autoref{theorem:lreg extension klr} it suffices to check the following:
\begin{enumerate}
\item The number of edges in $G'$ is in
\[
\Om{n^{v_{F^*_i}}p^{e_{F^*_i}}} \overset{\autoref{lemma:f star density}}{\gg} n^{2-1/\overline m^{r-i+1}_2(F) } \overset{\autoref{lemma:klr density}}{\geq} n^{2-1/m_2(e,F,-\log_n q)},
\]
since for $n$ large enough $-\log_n  q \leq 1/\overline m_2^{r-i+1}(F)-1/\overline m^r_2(F)$.
\item If we invoke the induction hypothesis with say $\kappa\leftarrow \kappa \eta_i\mu$ then $A_{i-1}[X,Y]$ is \lreg{F}{\mu}{q}.
\item To see that it is also \ue{F}{Aq} observe that every edge of $A_{i-1}$ spans a copy of of $(e,F'^*_-)\coloneqq\bigsqcup_{j<i} (e,F_-)\times (e,F^*_j)$. So for upper uniformity it suffices to bound the number of copies of $F'^*_-$ spanned by a single vertex. But our $p$ is such that we are above the $1$-density of $F'^*_-$ (\autoref{lemma:rooted density}). Thus this number is concentrated around its expectation (\autoref{theorem:spencer}), which is upper bounded by
\[
n^{v_{F'^*_-} - 1}p^{e_{F'^*_-}} = n \prod_{j<i} p_j \leq \abs X\frac{\prod_{j<i} p_j }{\kappa \prod_{j\leq i} \eta_j}   = Aq\abs X.
\]
\end{enumerate}

Therefore we can apply \autoref{theorem:lreg extension klr}  and $G'$ must come from a set of at most 
\[
\beta^{m} \binom{\abs X^2}{m}^{e_{F_-}} \leq \beta^{m} \left(\frac{e\abs X^2}{m}\right)^{e_{F_-}m} \leq  \beta^{m} \left(\frac{2e \tilde n_i^2}{m_i}\right)^{e_{F_-}m} 
\]
graphs. But then for our choice of $\beta$ the probability that $G'\subseteq G_i$ is in $\o{1}$.
\end{description}
\end{proof}
The proof of \autoref{lemma:dangerous} proceeds similarly to the proof of \autoref{claim:force regular claim} from the previous lemma. The only difference is that we replace $A_i$ with the hypergraph of $v_F$-tuples which span a copy of $F\times(e,F_-)$ in each of the graphs $G_i$.
\dangerous*
\begin{proof}
Let $t=v_{F}-1$ and let $k, \delta$ be such that \autoref{lemma:force regular} holds for all $S\subseteq \set{r}$, $\abs S \leq r-2$. Let $\eps_1,\dots,\eps_{r-1}$ denote constants whose value we will determine later in reverse order (that is $\eps_i$ will depend on $\eps_{i+1},\dots,\eps_{r-1}$.)

We ask Painter to color an instance of $G^k_{n,p}$. Applying \autoref{lemma:force regular} with $t\leftarrow t$, $S\leftarrow \emptyset$, $\eps\leftarrow \eps_1$ we obtain a constant $\eta_1(\eps_1)$, a graph $F^*_1\in \mathcal F^1$ and an $F^*_1$-spanning graph $G_1\in \G{K_t}{\tilde n_1}{m_1}{\eps_1}$ monochromatic in some color $s_1$. Pick one of the vertex partitions of $G_1$ arbitrarily and call it $V_1$. We now present Painter with a second instance of $G^k_{n,p}$ but only consider the subgraph induced by $V_1$ which is distributed like a $G^k_{\tilde n_1,p}$. Invoking \autoref{lemma:force regular} a second time with $\eps \leftarrow \eps_2$ and $S\leftarrow\left\{s_1\right\}$ we obtain a second graph $G_2\in \G{K_t,\tilde n_2, m_2,\eps_2}$. We repeat this procedure $r-1$ times and obtain
\begin{enumerate}
\item sets $V=V_0 \supset V_1 \supset \dots \supset V_{r-1}$ such that $\abs {V_i} \geq \eta_i \abs{V_{i-1}}$ for $i\in \set{r-1}$,
\item graphs $F^*_i \in \mathcal F^{\leq i}$ where $i\in \set{r-1}$,
\item monochromatic graphs $G_i\subseteq \G{K_t}{\tilde n_i}{m_i}{\eps_i}\subseteq G^{i\cdot k}_{n,p}[V_{i-1}]$ in pairwise different colors, where $\tilde n_i= \abs {V_i}$, $m_i\geq \eta_i \tilde n^{v_{F^*_i}}p^{e_{F^*_i}}$
such that $G_i$ is $F^*_i$-spanning in $G^{i\cdot k}_{n,p}[V_{i-1}]$.
\end{enumerate}
Furthermore for every graph $G'$ with $\omega(n)$ edges we have
\[
	\Pr{G'\subseteq G_i}\leq \left(\frac{m_i}{\delta\tilde n_i^2}\right)^{\abs{E\left(G'\right)}}.
\]
Observe that this probability is over the phases $(i-1)\cdot k+1,\dots,i\cdot k$ and that $G_i$ is fixed after the first $i\cdot k$ phases.

Let $A_0$ denote the complete directed $v_F$-uniform hypergraph on $V=V_0$. We identify the edges of $A_0$ with a (hypothetical) copy of $F$ in $V$ (depending on the automorphisms of $F$ different (directed) edges might represent the same copy of $F$). Now for $i\in \set{r-1}$ let $A_i$ denote the directed $v_F$-uniform hypergraph on $V_i\subseteq V_{i-1}$ where $e\in E(A_i)$ if $e\in E(A_{i-1})$ and additionally the edges of $e$ (when viewed as a graph $F(e)\sim F$) are the roots of pairwise edge disjoint copies of $F_-$ in $G_i$.

Define \begin{align*}
\tilde F^*_{i} &=  \bigsqcup_{j < i} (e,F)\times(e,F^*_j) \in \mathcal F^{i},\\
p_i &= n^{v_F-2}\left(n^{v_{F^*_i}-2}p^{e_{F^*_i}}\right)^{e_F-1}, \\
q_i &= \prod_{j\in \set{i}} p_i = n^{v_{\tilde F^*_{i+1}}-2} p^{e_{\tilde F^*_{i+1}}-1}.
\end{align*}
Since $G_i$ is $F^*_i$-spanning the edges of $F(e)$ (for $e\in A_i$) are not only roots of pairwise edge disjoint copies of $F_-$ but even of (still pairwise edge disjoint) copies of $F_-\times (e,F^*_i)$. Furthermore these copies are disjoint from those certifying membership in $A_1,\dots,A_{i-1}$. Thus for every $e\in E(A_i)$ the edges of $F(e)$ are the roots of pairwise edge disjoint copies of $\left(\tilde F_{i+1}\right)_-$.

\begin{claim}\label{claim:A_i upper-uniform}
Suppose that $X_1,\dots,X_{v_F}\subseteq V_i$ are mutually disjoint and of size $\tilde n$. Then $A_i[X_1,\dots,X_{v_F}]$, when viewed as a undirected hypergraph from $\mathcal R(V(F),\tilde n)$, is \ue{F}{(n/\tilde n)^{v_F} q_i} provided that $i\leq r-1$.
\end{claim}
\begin{proof}[Proof of claim] $q_0=1$ and thus the claim holds vacuously for $i=0$. For $i\geq 1$
 fix some $\sigma\subseteq V(A_i)$. If there exists $e$ such that $\sigma\subsetneq e \in E(A_i)$ then $\sigma$ fixes some $V'\subsetneq V(F)$ and the degree of $\sigma$ is at most the number of copies of $(V', F)\times(e, \left(\tilde F^*_{i+1}\right)_-)$ rooted in $\sigma$. By \autoref{lemma:rooted density} 
\[
	m(V', (V',F)\times(e,\left(\tilde F^*_{i+1}\right)_-)) < \overline m^{i+1}_2(F) \leq \overline m^r_2(F)
\]
 and thus by \autoref{theorem:spencer} this number is concentrated around its expectation which is at most $n^{v_F-\abs{V'}}q_i^{e_F-e_{F[V']}}$ as required for the upper-extensibility of $A_i$.
\end{proof}

\begin{claim} For every integer $i\leq r-1$ and $\kappa>0$ and small enough $\eps_1,\dots,\eps_i$ there exists $\gamma(\kappa, \eps_1,\dots,\eps_i)>0$ such a.a.s.\ for all pairwise disjoint equi-sized sets $X_1,\dots,X_{v_F}\subseteq_\kappa V_{i}$ the number of directed edges in $E\left(A_i\left[X_1,\dots,X_{v_F}\right]\right)$ is at least $\gamma \abs{X_1}^{v_F} q_i^{e_F}$.
\end{claim}
Invoking the claim for $i=r-1$ and say $\kappa = 1/(2v_F)$ proves the Lemma.

To prove the claim we will proceed by induction on $i$. For $i=0$ the statement holds vacuously with $\gamma = 1$ since $A_0$ is complete. For $i\geq 1$ we fix equi-sized and pairwise disjoint sets $X_1,\dots, X_{v_F}\subseteq_\kappa V_i$ of size $\tilde n$. Define
\[
m=\ceil{\frac{\tilde n^2 m_i}{2e_F\tilde n_i^2}}  \quad\quad\text{and}\quad\quad
\beta = \left(\frac{\delta}{4e_Fe}\right)^{e_F\cdot(e_F-1)}.
\]
Let $\mu=\mu(\beta)$ be given by \autoref{theorem:lreg extension klr} (invoked with $(R,F)\leftarrow (V(F),F\times(e,F_-))$) and let $\gamma'=\gamma(\kappa\eta_i\mu,\eps_1,\dots\eps_{i-1})$ denote the constant guaranteed by the induction hypothesis. Finally define
\[
A=\frac{1}{\gamma'\left(\kappa\prod_{j\in\set{i}}\eta_j\right)^{v_F}}.
\]

$F\times (e,F_-)$ is $t$-partite and thus for $\eps_i$ small enough depending on $\mu$ we can, using standard techniques, find a graph $G'_i\subseteq G_i$ from $\G{F\times (e,F_-)}{\tilde n}{m}{\mu}$ such that the vertex partitions corresponding to to the vertices of the (missing) inner copy of $F$ are the sets $X_1,\dots, X_{v_F}$.

If the number of edges in $T(G'_i, A_{i-1}[X_1,\dots,X_{v_F}])$ is at least
\[
\The{ \left(\frac{m}{\tilde n^2}\right)^{e_{F}\cdot (e_F-1)} \tilde n^{v_F + e_F\cdot (v_F-2)} q_{i-1}^{e_F} }= \The{p_{i}^{e_{F}}  n^{v_F}  q_{i-1}^{e_F}} = \The{n^{v_F}q_i^{e_F}},
\]
then we are done, since by \autoref{lemma:2-balanced building} $\left(V(F),F\times \left(e,F_-\right)\right)$ is balanced with density $m_2(F)$ and thus \autoref{lemma:span constant} the multiplicity of all edges is at most a constant.

Otherwise we invoke the induction hypothesis with $\kappa\leftarrow \kappa \eta_i \mu$ to deduce that  $A_{i-1}[X_1,\dots,X_{v_F}]$ is \lreg{F}{q}{\mu} where $q=\gamma'q_{i-1}$. We have
\[
\left(\frac{n}{\tilde n}\right)^{v_F} q_{i-1} \leq \frac{q_{i-1}}{ (\kappa \prod_{j\in\set{i}}\eta_j)^{v_F}} =Aq
\]
and therefore by \autoref{claim:A_i upper-uniform} the hypergraph $A_{i-1}[X_1,\dots,X_{v_F}]$ is \ue{F}{A q}. Finally
\begin{align*}
m &= \Om{n^{v_{F^*_i}}p^{e_{F^*_i}}} \overset{\autoref{lemma:f star density}}{\gg} n^{2-1/\overline m^{r-i+1}_2(F)} \overset{\autoref{lemma:klr density}}{\geq} n^{2-1/m_2(V(F),F\times(e,F),q_i)} \\
&\asymp n^{2-1/m_2(V(F),F\times(e,F),q)}
\end{align*}
and thus we may apply apply \autoref{theorem:lreg extension klr} with $q\leftarrow q$, $G\leftarrow G'_i$ and $G_R\leftarrow A_{i-1}[X_1,\dots,X_{v_F}]$ to deduce that $G'_i$  is from a set graphs of size at most
\begin{align*}
\beta^{m}\binom{\tilde n^2}{m}^{e_F\cdot(e_F-1)}
&\leq \beta^{m}\left(\frac{e\tilde n^2}{m}\right)^{m\cdot e_F\cdot(e_F-1)} \\
&\leq \beta^{m}\left(\frac{e2e_F\tilde n_i^2}{m_i}\right)^{m\cdot e_F\cdot(e_F-1)} \\
&= \left(\frac{\tilde n_i^2}{2\delta m_i}\right)^{m\cdot e_F\cdot(e_F-1)},
\end{align*}
Since $m\gg n$ a union bound over the $2^{\The{n}}$ choices for the sets $X_1,\dots,X_{v_F}$ together with the bound 
\[
\Pr{G'\subseteq G_i} \leq \left(\frac{m_i}{\tilde n_i^2 \delta}\right)^{\abs {E(G')}}
\]
from \autoref{lemma:force regular} guarantees that a.a.s.\ no such subgraph $G'_i$ exists.
\end{proof}

%% file: auxiliary.tex
% !Tex root = online_ramsey.tex

In \autoref{section:main proof} we stated a number of auxiliary statements without proof (namely \autoref{prop:FtimesFstar density}, \autoref{lemma:klr density} and \autoref{lemma:rooted density}). The proofs of these statements are somewhat technical and are given in this section.

We start with the proof of \autoref{lemma:klr density}  for which we need the following simple bound.
\begin{lemma}\label{lemma:m2<=2m}
Every graph $F$ on at least $3$ vertices and with at least one edge satisfies $m_2(F) \leq 2m(F)$.
\end{lemma}
\begin{proof}
One checks that the statement holds for all graphs on $3$ vertices.
If $F$ is $2$-balanced and contains at least $4$ vertices then
\[
m_2(F) = \frac{e-1}{v-2} \leq \frac{e}{v-2} \leq 2\frac{e}{v} = 2d(F) \leq 2m(F).
\]

Otherwise let $F'\subseteq F$ denote a graph that attains the $m_2$ density of $F$. By the above
\[
m_2(F) = m_2(F') \leq 2m(F') \leq 2m(F).
\]
\end{proof}
\klrdensity*
\begin{proof}Write $p=n^{-1/\overline m_2^k(F)}$ and $q = n^{-1/\overline m_2^k(F) + 1/\overline m^r_2(F)}$. The first inequality is equivalent to
\[
\min_{\substack{(R,F')\subseteq (e,F) \\ \overline e_{F'}\geq 1} } n^{v_{F'}-2} p^{\overline e_{F'}-1} q^{e_{F'[R]}} \geq 1.
\]

Fix such a rooted graph $(R,F')$.
If $\abs R\leq 1$ then $F'\subseteq F-\left\{e\right\}$ and $e_{F'[R]}=0$. Since $p \geq n^{-1/\overline m^2_2(F)} \geq n^{-1/m_2(F-\left\{e\right\})}$ we have
\[
 n^{v_{F'}-2} p^{e_{F'}-1} \geq 1.
\]

Otherwise $\abs R = 2$ and without loss of generality $e_{F'[R]}=1$ and $\overline e_{F'} = e_{F'} -1$. We rewrite the above as
\[
n^{v_{F'}-2} p^{e_{F'}-2}q =   \frac{n^{v_{F'}-2} p^{e_{F'}}}{pn^{-1/\overline m^r_2(F)}} \geq \frac{n^{v_{F'}-2} p^{e_{F'}}}{n^{-2/m_2(F)}}.
\]
By definition of $\overline m_2^{k}$ we have $n^{v_{F'}-2} p^{e_{F'}}\geq n^{-1/\overline m_2^{k-1}(F)}$. Together with
\[
n^{2/m_2(F)}  \overset{\autoref{lemma:m2<=2m}}{\geq}  n^{1/m(F)} \geq n^{1/\overline m^{k-1}_2(F)} 
\]
this implies the desired bound.

The second inequality is equivalent to
\[
\min_{\substack{(R,H)\subseteq (V(F),F\times (e,F)) \\ \overline e_H \geq 1}} n^{v_{H}-2} p^{\overline e_{H}} q^{e_{H[R]}} \geq 1.
\]

For $e\in E(F)$ let $F_e\subseteq H$ denote the graph isomorphic to a subgraph of $F$ which is attached to the root $e$ in $H$. There must exist at least on edge $e'$ such that $F_{e'}$ contains at least one non root edge. For such an edge we apply the first inequality to obtain
\[
n^{v_{F_{e'}}-2}p^{e_{F_{e'}}-e_{F_{e'}[e']}-1}q^{e_{F_{e'}[e']}} \geq 1.
\]
Thus the minimization is at least \[
n^{v_{H[R]-2}}\prod_{e\in E(F)\setminus \left\{e'\right\}} n^{v_{F_e}-v_{F_e[e]}} p^{e_{F_e}-e_{F_e[e]}}q^{e_{F_e[e]}}.
\]

If for some $e\in E(F)$ we have $v_{F_e[e]}\leq 1$ then $F_e$ is isomorphic to a subgraph of $F_-$ and
\[
n^{v_{F_e}-v_{F_e[e]}} p^{e_{F_e}-e_{F_e[e]}}q^{e_{F_e[e]}}
\geq n^{v_{F_e}-1} p^{e_{F_e}}\geq 1,
\]
since we are above the $1$-density of $F_-$. In particular if $H$ does not contain at least two root vertices then we are done. If $v_{F_e[e]}=2$ then
\begin{align*}
n^{v_{F_e}-v_{F_e[e]}} p^{e_{F_e}-e_{F_e[e]}}q^{e_{F_e[e]}}
&\geq n^{v_{F_e}-2} p^{e_{F_e}-1}q
= \frac{n^{v_{F_e}-2} p^{e_{F_e}}}{n^{-1/\overline m^r_2(F)}} \\
&\geq \frac{n^{-1/\overline m^{k-1}_2(F)}}{n^{-1/\overline m^r_2(F)}} 
\geq \frac{n^{-1/m(F)}}{n^{-1/m_2(F)}} 
\geq n^{-1/m_2(F)}.
\end{align*}

Thus the original minimization reduces to
\[
\min_{\substack{F'\subseteq F \\ v_{F'}\geq 2}} n^{v_{F'}-2}n^{-(e_{F'}-1)/m_2(F)}
\]
which is at least $1$ by definition of $m_2(F)$.
\end{proof}

 \autoref{prop:FtimesFstar density} concern the density of graphs in the class $\mathcal F^k$. Every graph $F^*\in\mathcal F^k$ (for $k\geq 2$) can be constructed by starting with a copy of $F$ and repeatingly attaching copies of $(e,F)$ to some edge. Since $F$ is $2$-balanced one may expect that graphs constructed by this procedue will be also be $2$-balanced. The following lemma establishes that this is indeed the case.
\begin{lemma}\label{lemma:2-balanced building}
Suppose that $G$ and $H$ are two $2$-balanced graphs such that $G\cap H$ is a single edge. If $G$ and $H$ both have $2$-density $d$ then $G\cup H$ is also $2$-balanced with density $d$.

Similarly if $(R,G)$ and $(R,H)$ are balanced rooted graphs of density $d$ with $V(G)\cap V(H) = R$ then $(R,G\cup H)$ is balanced with density $d$.
\end{lemma}
\begin{proof}
Let $p=n^{-1/d}$ and pick an induced subgraph $F\subseteq G\cup H$ with $e_{F}\geq 1$. Write $G'=F[V(G)]$ and $H'=F[V(H)]$. Without loss of generality we have $e_{G'}\geq 1$ and
\[
	n^{v_{F}-2}p^{e_{F}-1} = n^{v_{G'}-2}p^{e_{G'}-1} n^{v_{H'}-v_{H'\cap G'}}p^{e_{H'}-e_{H'\cap G'}} \geq 1,
\]
since $d=m_2(H)\geq m_1(H) \geq m(H)$ and since $H'\cap G'$ is either an edge, a vertex or empty. Thus $m_2(G\cup H) \leq d$. Furthermore \[
n^{v_{G\cup H}-2}p^{e_{G\cup H}-1} = n^{v_{G}-2}p^{e_G-1}n^{v_H-2}p^{e_H-1} = 1 \cdot 1,
\]
which implies $m_2(G\cup H)=d$ and that $G\cup H$ is balanced with respect to the $2$-density.

The second claim can be proved in a similar fashion.
\end{proof}
Thus every $F^*\in\mathcal F^k$, where $k\geq 2$, is $2$-balanced with $2$-density $m_2(F)$. Similarly $F\times (e,F^*)$ is $2$-balanced (and thus balanced)
\FtimesFstardensity*
\begin{proof}
Fix $F^*\in\mathcal F^r$, where $r\geq2$ . As noted above $F\times(e,F^*)$ is balanced. Therefore it suffices to check that
\begin{align*}
n^{v_{F\times(e,F^*)}}n^{-e_{F\times(e,F^*)}/\overline m_2^r(F)} 
&=n^{v_F}\left(n^{v_{F^*}-2}n^{-e_{F^*}/\overline m_2^r(F)}\right)^{e_F}  \\
&\overset{\mathclap{(\ref{lemma:f star density})}}{\geq}  n^{v_F} n^{-e_F/\overline m^1_2(F)} \geq 1.
\end{align*}
\end{proof}

It remains to prove \autoref{lemma:rooted density}. To do so we require two more auxiliary lemmas.
\begin{lemma}\label{lemma:prod density}
Let $(R,G),(e,H)$ be rooted graphs. Suppose that $(e,H)$ is balanced and that for some $t>0$
\begin{align*}
 m_1(H-e) &\leq t, \\
\overline v_H-\frac{\overline e_H}{t} &\geq  -\frac{1}{m(R,G)}.
\end{align*}
Then \[m\left(R,\left(R,G\right)\times \left(e,H-e\right)\right)\leq t.\]
\end{lemma}
\begin{proof}
Let $p=n^{-1/t}$ and let $(R,F)=(R,G)\times \left(e,H-e\right)$. It suffices to show that
\begin{equation}\label{eq:min prod}
\min_{(R,F')\subseteq (R,F)} n^{\overline v_{F'}-\abs R} p^{\overline e_{F'}} \geq 1.
\end{equation}

Fix a graph $(R,F')\subseteq(R,F)$ which attains the minimum.
Let $H_1,\dots,H_k \sim H-e$ denote the (canonical) copies of $H-e$ in $F$ and write $H'_i=F'\cap H_i$. The above term can be rewritten as
\[
n^{v_{F'\cap G}-\abs R} \prod_{i\in\set{k}} n^{v_{H'_i}- v_{H'_i\cap G}} p^{e_{H'_i}}.
\]

If for some $i$ we have $v_{H'_i\cap G}\in \left\{0,1\right\}$ then $t\geq m_1(H-e) \geq m(H-e)$ implies
\[
 n^{v_{H'_i}- v_{H'_i\cap G}} p^{e_{H'_i}} \geq 1.
\]
We may thus assume that for such $i$ we have $H'_i=H_i\cap G$.

  Otherwise $v_{H'_i\cap G}=2$. If $t\geq m(e,H)$ then the above bound holds as well and in particular (\ref{eq:min prod}) is satisfied. If $t<m(e,H)$ then, since $(e,H)$ is balanced, the minimum of $n^{v_{H'_i}-2}p^{e_{H'_i}}$ is attained for $H'_i=H_i$. Thus the minimization reduces to
\begin{align*}
\min_{(R,G')\subseteq (R,G)} n^{\overline v_{G'}} \left(n^{v_H-2}p^{e_H-1}\right)^{\overline e_{G'}}
&= \min_{(R,G')\subseteq (R,G)} n^{\overline v_{G'}} \left(n^{v_H-2-(e_H-1)/t}\right)^{\overline e_{G'}} \\
&\geq  \min_{(R,G')\subseteq (R,G)} n^{\overline v_{G'}} n^{-\overline e_{G'}/m(R,G)} = 1.
\end{align*}
\end{proof}

\begin{lemma}\label{lemma:forbidden density}
Let $r\geq 2$ and suppose that $G$ is a $2$-balanced graph with density at least $1$. Then
\[ 
	\max_{R\subsetneq V} m(R,G) \leq v-1 < \left(\frac{1}{m\left(G\right)}-\frac{1}{\overline m^r_2\left(G\right)}\right)^{-1}.
\]
\end{lemma}
\begin{proof}
$m(R,G)$ is monotone increasing under edge addition. Thus for the first inequality it suffices to consider the case $G=K_v$. We have
\[
	m(R,K_v)= \frac{\binom{v}{2}-\binom{\abs R}{2}}{v-\abs R} \leq \frac{\binom{v}{2}-\binom{v-1}{2}}{v-(v-1)} = v-1.
\]

For the second inequality we use $\overline m^r_2\left(G\right) < m_2(G)$ and the fact that since $G$ is $2$-balanced it is also balanced to obtain
\begin{equation}
\label{eq:forbidden max}
	\frac{1}{m\left(G\right)}-\frac{1}{\overline m^r_2\left(G\right)}
	< \frac{1}{m\left(G\right)}-\frac{1}{m_2\left(G\right)}
	= \frac{v}{e}-\frac{v-2}{e-1}.
\end{equation}
Maximizing (\ref{eq:forbidden max})  subject to $v \leq e$ we see that the maximum is attained whenever $v=e$. Thus the above is at most
\[
	1 - \frac{v-2}{v-1} = \frac{1}{v-1}.
\]
\end{proof}

\rooteddensity*
\begin{proof}
Let $F^*\in \mathcal F^k$. We have
\[
n^{v_{F^*}-1}n^{-e_{F^*}/\overline m_2^k(F)} \overset{\autoref{lemma:f star density}}{\geq} n^{1-1/\overline m^1_2(F)} \geq 1
\]
and thus $d_1(F^*) \leq \overline m^k_2(F)$. $F^*$ is $2$-balanced and thus strictly $1$-balanced. Therefore we obtain the inequality
\[
m_1(F^*_-) < m_1(F^*) = d_1(F^*) \leq \overline m^k_2(F),
\]
which proves the first part of \autoref{lemma:rooted density}.

For the second part we want to apply \autoref{lemma:prod density} with $(R,G)\leftarrow (V_0,F)$, $(e,H)\leftarrow (e,F^*)$ and $t\leftarrow \overline m^{k}_2(F)-\eps$, where $\eps > 0$ is a small constant such that $m_1(F^*_-)\leq  \overline m_2^k-\eps$. Since $F^*$ is $2$-balanced $(e,F^*)$ is also balanced. We have choosen $\eps$ such that $m_1(F^*_-) \leq t$. The final premise of \autoref{lemma:prod density} is established by
\[
v_{F^*}-2 - \frac{e_{F^*}-1}{\overline m^{k}_2(F)}
\overset{\autoref{lemma:f star density}}{\geq} -\frac{1}{\overline m^{1}_2(F)}+\frac{1}{\overline m^{k}_2(F)}
\overset{\autoref{lemma:forbidden density}}{>} -\frac{1}{m(V_0, G)}.
\]
Thus we can apply \autoref{lemma:prod density} which proves the last property.

\end{proof}

%\begin{lemma}\label{lemma:1 balanced}
%Every connected $2$-balanced graph that is not a tree is strictly balanced and $1$-balanced. Furthermore $m_1(G)\leq \overline m^2_2(G)$.
%\end{lemma}
%\begin{proof}
%\[
%\frac{e-e'}{v-1-(v'-1)} = \frac{(e-1)-(e'-1)}{(v-2)-(v'-2)} \geq \frac{e-1}{v-2} = \frac{e-1}{(v-1)-1} > \frac{e}{v-1}.
%\]
%\end{proof}

%% file: partI.tex
% !Tex root = online_ramsey.tex

The proof of the theorem follows the proof of the K\L R-conjecture by Saxton, Thomason in \cite{SaxtonThomason:container} and relies on their container theorem:
\begin{definition}
Let $G$ be an $r$-graph of order $n$ and average degree $d$. Let $\tau>0$. Given $v\in V(G)$ and $2\leq j \leq r$, let 
\[
	d^{(j)}(v) = \max \left \{ d(\sigma) : v\in \sigma \subset V(G), \abs \sigma = j \right\}.
\]
If $d > 0$ we define $\delta_j$ by the equation
\[
	\delta_j \tau^{j-1}nd=\sum_v d^{(j)}(v).
\]
Then the co-degree function $\delta(G,\tau)$ is defined by
\[
	\delta(G,\tau) = 2^{\binom{r}{2}-1}\sum_{j=2}^r \delta_j 2^{-\binom{j-1}{2}}.
\]
If $d=0$ we define $\delta(G,\tau)=0$.
\end{definition}

\begin{theorem}[\cite{SaxtonThomason:container}, Corollary 3.6]
\label{theorem:container}
Let $\mathcal E$ be an $r$-graph on the vertex set $\set{n}$. Let $0<\eps,\tau<1/2$. Suppose that $\tau$ satisfies $\delta(\mathcal E,\tau) \leq \eps/12r!$. Then there exists a constant $c=c(r)$, and a function $C\colon \mathcal P\left([n]\right)^s \to \mathcal P[n]$ where $s\leq c \log(1/\eps)$, with the following properties. Let
$\mathcal T= \left\{\left(T_1,\dots,T_S\right)\in\mathcal P\left(\set{n}\right)^s\colon \abs {T_i} \leq c\tau n\right\}$, and let $\mathcal C = \left\{C(T)\colon T\in\mathcal T\right\}$. Then
\begin{enumerate}
\item for every $I\subset\set{n}$ for which $e(\mathcal E[I])\leq \eps \tau^r e(\mathcal E)$ there exists $T=\left(T_1,\dots,T_s\right)\in \mathcal T \cap P\left( I\right)^s$ with $I\subset C(T)$,
\item $e(\mathcal E[C])\leq \eps e(\mathcal E)$ for all $C\in\mathcal C$.
\end{enumerate}
\end{theorem}

For a graph $F$ we denote with $K_{F,n}$ the $v_F$-partite graph with vertex partitions $V_1,\dots,V_{v_F}$ of size $n$, such that $K_{F,n}[V_i,V_j]$ is complete if $\left\{i,j\right\}\in E(F)$ and empty otherwise. 

For a rooted graph $(R,F)$ and $G_R\in \mathcal R(R,n)$ we denote with $\mathcal E\left(G_R,F\right)$ the hypergraph whose vertices are the edges of $K_{F_-,n}$ and whose edges form (when seen as subgraphs of $K_{F_-,n}$) a partite copy of $F_-$ whose roots induce an edge in $G_R$.

To proof \autoref{theorem:lreg extension klr} we will apply \autoref{theorem:container} to $\mathcal E(G_R,F)$. The first step is to obtain a bound on the co-degree function.
\begin{lemma}\label{lemma:delta(E,gamma) bound}
Let $\left(R,F\right)$ be a rooted graph with $e_{F_-}>1$. Let $0<\gamma, q(n) \leq 1 \leq A$.

Then for $n$ sufficiently large every hypergraph $G_R\in\mathcal R(R,n)$ which is \ue{F}{Aq} satisfies
\[\delta\left(\mathcal E\left(G_R,F\right),\gamma^{-1}n^{-1/m_2\left(R,F,-\log_n\left(q\right)\right)}\right)\leq \gamma e_{F_-}2^{e_{F_-}^2}\frac{n^{\abs R}\left(Aq\right)^{e_{F[R]}}}{\abs{E(G_R)}}.\]
\end{lemma}
\begin{proof}
Let $\sigma$ denote a set of vertices of $\mathcal E=\mathcal E(G_R, F)$. We identify $\sigma$ with the set of edges from $K_{F_-,n}$ which it represents. If the degree of $\sigma$ is non zero this set of edges is a graph $F'\subset K_{F_-,n}$ which is isomorphic to some subgraph of $F_-$. The degree of $F'$ is the number of ways we can extend $F'$ to a partite copy of $F_-$ in $K_{F_-,n}$ whose roots form an edge in $G_R$.

Since $G$ is \ue{F}{Aq} we have 
\[
	d\left(F'\right)\leq n^{v_F-v_{F'}}\left(Aq\right)^{e_{F[R]}-e_{F[R\cap V(F')]}}  .
\]
For $j\geq 2$ and an edge $e\in E\left(K_{F_-,n}\right)$ the quantity $d^{(j)}\left(e\right)$ is the maximum of $d\left(F'\right)$ over all $F'$ with $e\in F'$ and $\abs{F'}=j$. Thus
\[
d^{(j)}\left(e\right) \leq n^{v_F-v_{F_j}}\left(Aq\right)^{e_{F[R]}-e_{F[R\cap V(F_j)]}}
\]
where
\[
	F_j = \argmax_{\substack{F'\subseteq F \\ e\left(F'\right)-e\left(F'[R]\right)=j}} n^{-v_{F'}}\left(Aq\right)^{-e_{F'[R]}}.
\]
Observe that $F_j[R] = F[R\cap V(F_j)]$.
Let $t=-\log_n\left(q\right)$ and $\tau = \gamma^{-1}n^{-1/m_2\left(R,F,t\right)}$. Using $m_2\left(R,F,t\right)\geq d_2\left(R\cap V(F_j), F_j,t\right)$ we obtain
\begin{align*}
\frac{1}{\tau^{j-1}}
&=\gamma^{j-1}\left(n^{1/m_2\left(R,F,t\right)}\right)^{(j-1)} \\
& \leq \gamma^{j-1}\left(n^{1/d_2\left(R\cap V(F_j),F_j,t\right)}\right)^{(j-1)} \\
&= \gamma^{j-1} n^{v_{F_j}-2}q^{e_{F_j[R]}}.
\end{align*}
The number of edges in $\mathcal E$ is $\abs{E(G_R)} n^{v_F-\abs R}$. Thus for $j\geq2$ we have
\[
\delta_j = \frac{\sum_{e}d^{(j)}\left(e\right)}{\tau^{j-1}e_{F_-}\abs{E\left(\mathcal E\right)}}
	\leq \frac{e_{F_-} n^2 n^{v_F-v_{F_j}}\left(Aq\right)^{e_{F[R]}-e_{F_j[R]}}}{\tau^{j-1}e_{F_-} \abs{E(G_R)} n^{v_F-\abs R}}
	\leq \gamma^{j-1} \frac{n^{\abs R}A^{e_{F[R]}-e_{F_j[R]}}q^{e_{F[R]}}}{\abs{E(G_R)}}.
\]
Finally we obtain
\[
	\delta\left(\mathcal E,\tau\right)=2^{\binom{e_{F_-}}{2}-1}\sum_{j=2}^{e_{F_-}}\delta_j2^{-\binom{j-1}{2}} \leq e_{F_-}2^{e_{F_-}^2}\gamma \frac{n^{\abs R}\left(Aq\right)^{e_{F[R]}}}{\abs{E(G_R)}}.
\]
as claimed.
\end{proof}

Having bounded the co-degree function we can obtain a collection of containers for $\mathcal E\left(G_R,F\right)$ which do not induce many edges in $\mathcal E\left(G_R,F\right)$. Viewing our contains as subgraphs of $K_{F_-,n}$ this means that they contain few copies of $F_-$ whose roots induce an edge in $G_R$. To prove a K\L R-type statement we want our containers to be sparse subgraphs of $K_{F_-,n}$. The following two lemmas establish that if $G_R$ is lower-regular then the containers obtained by \autoref{theorem:container} are indeed sparse.

\begin{lemma}\label{lemma:dense extensions} Let $(R,F)$ denote a rooted graph. For every $\delta>0$ there exists $\epsilon>0$ such that for all $p\geq\delta$ the following holds.
Suppose that $G_{R}\in \mathcal R(R, n)$ is \lreg{F}{q}{\eps} and that the bipartite graphs of $G\subseteq K_{F_-,n}$ are \reg{\eps} with density at least $p$ then 
\[ \abs{E\left(T\left(G, G_R\right)\right)} \geq \left(1-\delta\right)p^{e_{F_-}} q^{e_{F[R]}} n^{v_F}.\]
\end{lemma}
\begin{proof}
Observe that  the density $p$ is at least $\delta$, which is a constant. Therefore $G$ is a (dense) regular graph and standard counting arguments apply. We only sketch of the proof: using standard arguments we find roughly $n^{v_F-\abs R}p^{e_{F\left[V\left(F\right)\setminus R\right]}}$ tuples in $\bigtimes_{i\in V(F)\setminus R} V_i$ whose common neighborhoods into the partitions $G_R$ are roughly as large as expected (in particular they are of linear size). Thus for $\eps$ small enough the \lreg{F}{q}{\eps}ity of $G_R$ guarantees that every one of these tuples extends to roughly $n^{\abs R}p^{e_F-e_{F\left[V\left(F\right)\setminus R\right]}}q^{e_{F[R]}}$ copies of $F_-$ whose roots form an edge in $G_R$. Every such copy of $F_-$ contributes on edge to the multi-hypergraph $T(G, G_R)$ and we obtain the desired bound.
\end{proof}

\begin{lemma}\label{lemma:extension containers}
Let $(R,F)$ be a rooted graph with $e_{F_-}>1$. Let $\delta>0$ be small enough and let $A\geq 1$. Then there exists $c, \eps(\delta), R(\delta), \gamma(\delta, A)$ such that the following is true. Suppose that $\tau(n), q(n)\in\o{1}$ satisfy $\tau \geq \gamma^{-1}n^{-1/m_2\left(R,F,-\log_n q\right)}$. If $G_R\in \mathcal R(R,n)$ is \ue{F}{Aq} and \lreg{F}{q}{\eps} then for $n$ large enough there exists a collection $\mathcal C$ of subgraphs of $K_{F_-,n}$ such that
\begin{enumerate}
\item for every $G\subseteq K_{F_-,n}$ for which $e\left(T\left(G, G_H\right)\right) \leq \eps \tau^{e_{F_-}} q^{e_{F[R]}} n^{v_F}$ there exists $T_1,\dots,T_s\subseteq G$ with $G\subset C(T_1,\dots,T_s)\in\mathcal C$, $e\left(T_i\right)\leq c \tau n^{2}$ and $s\leq c \log (A^{e_{F[R]}}/\eps)$,
\item for every $C\in\mathcal C$ there exists $\{i,j\}\in E(F_-)$ and equitable partitions $V_i=V_{i,1}\cup\dots\cup V_{i,r}$ and $V_j=V_{j,1}\cup\dots\cup V_{j,r}$ where $r\leq R\left(\delta\right)$ such that for at least $r^2/2e_{F_-}$ pairs $x,y\in \set{r}$ we have  $e\left(C\left[V_{i,x},V_{j,y}\right]\right)\leq \delta \abs{V_{i,x}}\abs{V_{j,y}}$.
\end{enumerate}
\end{lemma}
\begin{proof}
The constants $c,\mu(\delta), R(\mu), \eps(\delta, \mu, R)$ and $\gamma(\eps, A)$ will be determined later.
Let $\eps'=\eps q^{e_{F[R]}}n^{\abs R}/e\left(G_R\right)$ and $\mathcal E=\mathcal E\left(G_R,F\right)$. Since $G_R$ is \ue{F}{Aq} we can invoke  \autoref{lemma:delta(E,gamma) bound} to obtain the bound
\[ \delta\left(\mathcal E, \tau\right) \leq  \gamma e_{F_-} 2^{e^2_{F_-}} \frac{n^{\abs R}\left(Aq\right)^{e_{F[R]}}}{e(G_H)} = \gamma e_{F_-} 2^{e_{F_-}^2} \frac{\eps' A^{e_{F[R]}}}{\eps}.\]
For $\gamma(\eps, A)$ small enough we obtain
\[
\delta\left(\mathcal E, \tau\right)  \leq \frac{\eps'}{12 r!},
\]
which is what we need to apply \autoref{theorem:container} with $\mathcal E\leftarrow\mathcal E$, $\eps\leftarrow \eps'$, $\tau \leftarrow \tau, r\leftarrow e_{F_-}$ to obtain a collection of containers $\mathcal C$. We will now show that these containers (when viewed as subgraphs of $K_{F_-,n}$) satisfy the conditions of our Lemma.

So let $G\subseteq K_{F_-,n}$ with 
\[e\left(T\left(G, G_R\right)\right)\leq \eps\tau^{e_{F_-}} q^{e_{F[R]}} n^{v_F} = \eps' \tau^{e_{F_-}}n^{v_F-\abs R} e(G_R) = \eps' \tau^{e_{F_-}}  e\left(\mathcal E\right).\]
Define $I=E\left(G\right)$ and observe that $e\left(\mathcal E\left[I\right]\right)=e\left(T\left(G,G_R\right)\right)$ and thus $e\left(\mathcal E\left[I\right]\right) \leq \eps' \tau^{e_{F_-}}e\left(\mathcal E\right)$. Therefore we obtain $T_1,\dots,T_s\subseteq G\subseteq C\left(T_1,\dots,T_s\right)\in\mathcal C$ with $e\left(T_i\right) \leq c'\tau v\left(\mathcal E\right) = c\tau n^{2}$ and $s\leq c \log\left(1/\eps'\right)$. Since $G_R$ is \ue{F}{Aq} we have
\[
\log \left(\frac{1}{\eps'}\right) = \log\left(\frac{e(G_R)}{\eps q^{e_{F[R]}}n^{\abs R}}\right) \leq \log\left(\frac{A^{e_{F[R]}}}{\eps}\right),
\]
which proves the bound on $s$.

It remains to show that every $C\in\mathcal C$ contains a sparse partition (in the sense of 2.). To this end, for $\mu$ small enough depending on $\delta$, consider a $\mu$-regular partition of $C$ which refines the initial partition $\left(V_i\right)_{i\in V(F)}$. For every $i$ we obtain a partition $V_i=V_{i,1}\cup\dots\cup V_{i,r}$ for some $r\leq R\left(\mu\right)$. Now consider $C_x= C\left[V_{1,x_1}\cup\dots\cup V_{v_{F},x_{v_F}}\right]$ for some $x\in \set{r}^{v_F}$. For $\eps$ small enough depending on $R,\mu$ the $\abs R$-graph $G_{R,x}=G_R[V\left(C_x\right)]]$ is \lreg{F}{\mu}{q}. Thus if all pairs in $C_x$ are $\mu$-regular with density at least $\delta$ then by  \autoref{lemma:dense extensions} for $\mu$ small enough depending on $\delta$ 
\[ e\left(T\left(C_x, G_{R,x}\right)\right) \geq \left(1-\delta\right)\delta^{e_{F_-}}q^{e_{F[R]}}\left(\frac{n}{2r}\right)^{v_F}.\]
But $e\left(T\left(C_x, G_{R,x}\right)\right)$ is at most
\[
e\left(\mathcal E[C]\right)\leq \eps' e\left(\mathcal E\right) = \eps' n^{v_F-\abs R}e\left(G_R\right) = \eps n^{v_F}q^{e_{F[R]}},
\]
which is a contradiction for $\eps$ small enough depending on $R, \delta$.

Thus for every $x$ there exists $\{i,j\}\in E(F)$ such that $C\left[V_{i,x_i},V_{j,x_j}\right]$ is either sparse or not $\mu$-regular. By the pigeonhole principle at least an $1/e_F$-fraction of the $x$ nominate the same edge $\left\{i,j\right\}$ and every pair $V_{i,a}, V_{j,b}$ is nominated by at most $r^{v_F-2}$ different $x$. Finally at most an $\mu$-fraction of these pairs is not $\mu$-regular. Therefore we have found $i,j$ such that at least $r^2/(2e_{F_-})$ of the pairs $V_{i,\cdot}, V_{j,\cdot}$ have density at most $\delta$.
\end{proof}

The proof of \autoref{theorem:lreg extension klr} now follows from a standard counting argument:
\extensionklr*
\begin{proof}
The proof will require a number of constants which will be fixed during the proof. Their dependencies are as follows:
$\delta(\beta)$, $\eps(\delta)$,  $R(\delta)$, $\gamma(\delta, A)$, $\hat s(A, \eps)$, $\eta(\hat s,\gamma)$, $\alpha(\eps, \gamma, \eta)$, $\mu(\eps, R)$.

We invoke \autoref{lemma:extension containers} with $\delta\leftarrow\delta$, $A\leftarrow A$ and obtain constants $c$, $\eps\left(\delta\right)$, $R\left(\delta\right)$ and $\gamma\left(\delta,A\right)$.

Fix  $\tau=\eta m/n^2$. For $\alpha$ small enough depending on $\gamma$ and $\eta$ we have $\tau\geq \gamma^{-1} n^{-1/m_2\left(R,F,-\log_n q\right)}$ and for $\alpha$ small enough depending on $\eps$ and $\eta$ we have
\[
	\eps \tau^{e_{F_-}}q^{e_{F[R]}}n^{v_F} \geq \alpha(m/n^2)^{e_{F_-}}q^{e_{F[R]}}n^{v_F}.
\]
Therefore for $\mu\leq\eps$ small enough \autoref{lemma:extension containers} guarantees the existence of a container $T_1,\dots,T_s\subseteq G \subseteq C\left(T_1,\dots,T_s\right)$ with $s\leq c \log\left(A^{e_{F[R]}}/\eps\right)\eqqcolon\hat s$ whenever $G\in\G{F_-}{n}{m}{\mu}$ does not satisfy $e\left(T\left(G, G_R\right)\right) > \alpha \left(m/n^2\right)^{e_{F_-}}q^{e_{F[R]}}n^{v_F}$.

To count all such graphs $G$ we fix $T=\left(T_1,\dots,T_s\right)$ and then pick $G\subseteq K_{F_-,n}$ u.a.r.\ among all graphs with exactly $m$ edges in each bipartite graph.
Following \cite{SaxtonThomason:container} we define the following events
\begin{align*}
	E_T&\colon T_1\cup\dots\cup T_s\subseteq G \subseteq C(T) \text{ and $G\in\G{F_-}{n}{m}{\mu}$,} \\
	F_T&\colon T_1\cup\dots\cup T_s\subseteq G,\\
	G_T&\colon G \subseteq C(T) \text{ and $G\in\G{F_-}{n}{m}{\mu}$.}
\end{align*}

We firstly show that $\sum_T \Pr{F_T}\leq 2^m$. Note that this is the expected number of tuples $T\subseteq G$. The maximum number of tuples $T\subseteq G$ is at most
\begin{align*}
\sum_{\abs{T_1},\dots,\abs{T_{\hat s}}} \prod_{i\leq \hat s} \binom{e_{F_-} m}{\abs{T_i}}
&\leq \left(c\tau n^{2}\right)^{\hat s}\binom{e_{F_-}m}{c\tau n^{2}}^{\hat s} \\
&\leq \left(c\tau n^{2}\right)^{\hat s}\left(\frac{e e_{F_-}m}{c\tau n^{2}}\right)^{\hat s c\tau n^2}  \\
&= \left(c\eta m\right)^{\hat s}\left(\frac{e e_{F_-}}{c\eta}\right)^{\hat s c\eta m}  \\
&\leq 2^m,
\end{align*}
for $\eta$ small enough depending on $\hat s$.

Secondly we show that $\Pr{G_T\mid F_T}\leq \left(\beta/2\right)^m$ and thus
\[
\sum_T \Pr{E_T}  = \sum_T\Pr{G_T \mid F_T} \Pr{F_T} \leq \beta^m,
\]
which implies the Theorem.

For fixed $T$ and $C\left(T\right)$ let $\{i,j\} \in E(F_-)$ and $V_i=V_{i,1}\cup\dots\cup V_{i,r}$ and $V_j=V_{j,1}\cup\dots\cup V_{j,r}$ be given by property (2) of \autoref{lemma:extension containers}. For $\mu\left(R\right)$ small enough we use the \reg{\mu}ity of  $G\left[V_i,V_j\right]$ to require \[
\abs{G\left[V_{i,x},V_{j,x}\right]}\geq \left(1-\mu\right)\left(\frac{n}{r}\right)^2 \frac{m}{n^2} \geq m/2r^2
\]
for every $x,y\in\set{r}$ while $\abs{C\left[V_{i,x},V_{j,x}\right]}\leq \delta\left(\frac{n}{r}\right)^2$ for at least $r^2/2e_{F_-}$ choices of $x,y$. Let $C'=\bigcup C\left[V_{i,x},V_{j,x}\right]$ where the union runs over $r^2/2e_{F_-}$ sparse pairs. We have $\abs{C'}\leq\delta n^2/2e_{F_-}$ and for $G$ to be \reg{\mu} we require $\abs{G\cap C'}\geq m/4e_{F_-}$. We conclude for $\eta\left(\gamma,c\right)$ small enough
\begin{align*}
	\Pr{G_T\mid F_T}
&\leq \Pr{\abs{G\cap C'}\geq m/4e_{F_-}\mid F_T} \\
&\leq \Pr{\abs{\left(G-T\right)\cap C'}\geq m/4e_{F_-}-\hat sc\tau n^2} \\
&\leq \Pr{\abs{\left(G-T\right)\cap C'}\geq m/6e_{F_-}} \\
&\leq \binom{\delta \frac{n^2}{2e_{F_-}}}{\frac{m}{6e_{F_-}}} \left(\frac{m}{n^2}\right)^{m/6e_{F_-}} \leq \left(3e\delta\right)^{m/6e_{F_-}} \leq \left(\frac{\beta}{2}\right)^{m},
\end{align*}
for $\eta(\hat s)$ and $\delta(\beta)$ small enough.
\end{proof}